\documentclass[12pt]{article}
\usepackage{graphicx}
\usepackage{color}
 \definecolor{rouge}{cmyk}{0,1,1,0.4}
\usepackage{amssymb,amsmath,amsfonts,amsthm}

\newtheorem{theorem}{Theorem}[section]

\newtheorem{theo}[theorem]{Theorem}

\newtheorem {hypo}[theorem]{Hypothesis}

\newtheorem {pro}[theorem]{Proposition}
\newtheorem{defi}[theorem]{Definition}

\newtheorem{lem}[theorem]{Lemma}

\newdimen\AAdi%
\newbox\AAbo%
\def\bsi{\overline{\sigma}}
\def\hsi{\widehat{\sigma}}
\def\bet{\overline{\eta}}
\def\bx{{\overline{x}}}
\def\AArm{\fam0 }
\def\AAk#1#2{\setbox\AAbo=\hbox{#2}\AAdi=\wd\AAbo\kern#1\AAdi{}}%
\def\AAr#1#2#3{\setbox\AAbo=\hbox{#2}\AAdi=\ht\AAbo\raise#1\AAdi\hbox{#3}}%
\def\cC{{\mathcal C}}%
\def\BBp{{\AArm I\!P}}%
\def\BBr{{\AArm I\!R}}%
\def\BBone{{\AArm 1\AAk{-.8}{I}I}}%

\usepackage[ansinew]{inputenc}

\begin{document}
\title{Dynamical coupling between Ising and FK percolation}
\author{
\qquad Rapha\"el Cerf\footnote{
\noindent
DMA, Ecole Normale Sup\'erieure,
CNRS, PSL Research University, 75005 Paris.}
\footnote{
\noindent Laboratoire de Math\'ematiques d'Orsay, Universit\'e Paris-Sud, CNRS, Universit\'e
Paris--Saclay, 91405 Orsay.}
\hskip 60pt
Sana Louhichi \footnote{Laboratoire Jean Kuntzmann.
B\^atiment IMAG - 700 avenue centrale - 38400 Saint Martin d'H\`eres, France.
}
}
\maketitle
\abstract{
We investigate the problem of constructing a dynamics on edge--spin configurations
which realizes a coupling between a Glauber dynamics of the Ising model and a
dynamical evolution of the percolation configurations.
We dream of constructing a Markov process on edge--spin configurations which is reversible with
respect to the Ising--FK coupling measure, and such that the marginal on the spins is a 
Glauber dynamics, while the marginal on the edges is a Markovian evolution.
We present two local dynamics, one which fulfills only the first condition and one
which fulfills the first two conditions.
We show next that our dream process is not feasible in general.
We present a third dynamics, which is non local and fulfills the first and the third conditions.
We finally present a localized version of this third dynamics, which 
can be seen as a contraction of the first dynamics.
}
\section{Introduction}
The Ising model, invented almost a century ago, is one of the most studied model of statistical mechanics.
A fundamental tool to analyze the Ising model is the random cluster model, or FK model, invented by Fortuin and Kasteleyn around 1969 (see the reference book \cite{Grimmett}).
Results on this
dependent percolation model can be transferred towards the Ising model via a coupling construction due to
Edwards and Sokal \cite{Sokal}.
This machinery works extremely well, but so far, it has essentially been employed to study the systems
at equilibrium.
Yet another rich facet of the Ising model is the dynamics. There exist several microscopic dynamics
on spin configurations, which give rise to the so--called stochastic Ising model, and
whose equilibrium are described by the Ising Gibbs measure:
the Metropolis dynamics, 
the heat--bath dynamics, the Kawasaki dynamics, to name a few of them.
On one hand, these dynamics provide a basic model
to study fundamental questions on dynamics, for instance to model the metastability phenomenon \cite{OS}.
On the other hand, there is a hope that the understanding of the dynamics will shed light on
the equilibrium measure.

Percolation models possess also a natural dynamics. For the Bernoulli percolation model, it consists
in updating independently the edges, and it leads to beautiful difficult problems \cite{Olle,Steif}.
For the FK model, the construction of a dynamics is more subtle and it involves typically non--local
computations \cite{SCM}.
One naturally wonders whether the stochastic Ising model and dynamical percolation could help to
understand each other.
Our note is a little investigation into the possibility of building a coupling for the dynamics
on Ising and percolation models, which could help to understand the dynamics of the Ising model.
We focus here on the case of a Glauber type dynamics, that is a local dynamics which modifies at 
most one spin at a time. The precise definition of such dynamics, together with several examples,
are given in section~\ref{glauber}.
We dream of building a coupling  dynamics on edge--spins configurations such that:
\smallskip

\noindent
$\bullet$ The marginal on the spins is a Glauber dynamics.

\noindent
$\bullet$ The marginal on the edges is a simple Markovian evolution.

\noindent
$\bullet$ The coupling dynamics on edge--spins configurations is reversible with respect to the
coupling measure between the Ising and the FK models.
\smallskip

\noindent
Notice that if we drop the first constraint, then 
we could simply 
consider a time continuous Markov process on the edge--spins configurations which, after
an exponential time of parameter one, jumps on a new independent edge--spins configuration,
drawn according to the Ising--FK coupling measure. However the first constraint is essential
for us, indeed our hope is to build a dynamical coupling which would help to study
a Glauber dynamics of the Ising model.

The definition of the Ising--FK coupling measure is recalled in section~\ref{IFK}.
The stochastic Ising model is defined in section~\ref{glauber}. The FK dynamics is defined in
section~\ref{fkdyna}.
The above conditions are precisely stated in section~\ref{goal}.
Unfortunately, we did not succeed in building a dynamics which fulfill these three conditions,
even if we weaken partially the last one.
However, we manage to build couplings which satisfy the third condition. The most basic coupling is a dynamics which changes at most one object
at each step, i.e., either one spin or one edge is modified at a time. We prove that, for this type of coupling,
the first condition can never be satisfied.
We build such a coupling
dynamics in section~\ref{one}.
Another natural dynamics consists in changing simultaneously the spin at one vertex,
together with the edges incident to this vertex.
Such a dynamics is presented in section~\ref{two}.
The good news is that the
marginal of this dynamics on the spins is a Glauber dynamics.
However the marginal of this type of dynamics on the edges is not Markovian.
%
%
In section~\ref{general},
we show that our dream process is not feasible in general.
We present a third dynamics in section~\ref{three}, 
which is non local and fulfills the first and the third conditions.
We finally present in section~\ref{four} a localized version of this third dynamics, which 
can be seen as a contraction of the first dynamics.
\section{Ising--FK coupling measure}
\label{IFK}
In this section, we recall some classical notation and we define the Ising--FK coupling measure~$IP$.
We consider a finite graph $(V,E)$, where $V$ is the set of the vertices and $E$ the set of the edges.
The edges are unoriented, and $E$ is a subset of the set of pairs of points of $V$.
Throughout the paper, we suppose that the set of edges $E$ is not empty.
An edge configuration $\eta$ is an element of
$\{0,1\}^E$ (where $0$ stands for closed and $1$ for open).
A spin configuration $\sigma$ is an element of
$\{-1,1\}^V$.
An edge $e$ with endvertices $x$ and $y$
is written
as $e=\langle x,y\rangle$
or $e=\langle y,x\rangle$.
%
We define
$$\forall\sigma\in \{-1,1\}^V
\quad\forall e=\langle x,y\rangle\in E\qquad \delta_{\sigma}(e)=\BBone_{\sigma(x)=\sigma(y)} \,.$$
An edge configuration $\eta$ and a spin configuration $\sigma$ are said to be compatible if we have
$$\forall e\in E\qquad
\eta(e) \leq \delta_{\sigma}(e)
\,,$$
i.e., if the endvertices of any open edge in $\eta$ have the same spins in $\sigma$.
We denote by $\cC$ the set of the pairs of compatible configurations:
$$
\cC=\big\{(\eta,\sigma)\in \{0,1\}^E\times \{-1,1\}^V: \eta(e) \leq \delta_{\sigma}(e)\,\,\,\text{for any}\,\,\,
e\in E\big\}.
$$
Let $p$ be a fixed number in $[0,1]$.
The Ising-Percolation measure $IP$ is the probability measure  on the product space $\{0,1\}^E\times \{-1,+1\}^V$
defined as follows.
For any $(\eta,\sigma)$ in $\{0,1\}^E\times\{-1,+1\}^V$, we set
$$
{IP}(\eta,\sigma)= \frac{1}{Z} \prod_{e\in E}\left(p\BBone_{\eta(e)=1}\delta_{\sigma}(e)+ (1-p)\BBone_{\eta(e)=0} \right),
$$
where $Z$ is the partition function, i.e., the normalising constant that makes ${IP}$ a probability measure on $\{0,1\}^E\times\{-1,+1\}^V$,  given by
$$
Z=\sum_{(\eta,\sigma)\in \{0,1\}^E\times\{-1,+1\}^V}\prod_{e\in E}\left(p\BBone_{\eta(e)=1}\delta_{\sigma}(e)+ (1-p)\BBone_{\eta(e)=0} \right).
$$
The Ising-Percolation measure can also be written as
\begin{equation}\label{ipsimply}
{IP}(\eta,\sigma)= \frac{1}{Z} p^{|\{e\in E:\, {\eta(e)=1}\}|} (1-p)^{|\{e\in E:\, {\eta(e)=0}\}|}\BBone_{(\eta,\sigma)\in \cC},
\end{equation}
where the absolute value of a set denotes its cardinality. In fact,
the measure $IP$ is the Bernoulli product measure on edge-spin configurations with parameters $(p,1/2)$
conditioned to the set $\cC$ of the compatible configurations.
It is well--known that the probability measure ${IP}$ is a coupling of the Ising measure $\mu_{\beta}$
defined on $\{-1,+1\}^V$,
together with the random-cluster measure $\phi_{p,2}$ defined on $\{0,1\}^E$, with the relation $p=1-e^{-\beta}$. In fact, the first marginal measure $\phi_{p,2}$ on $\{0,1\}^E$ is given by
\begin{equation}\label{tramsf}
\forall\,\,\eta\in \{0,1\}^E\qquad \phi_{p,2}(\eta)=\sum_{\sigma \in \{-1,+1\}^V}{IP}(\eta,\sigma)= \frac{1}{Z_{RC}}
\left(\prod_{e\in E} p^{\eta(e)}(1-p)^{1-\eta(e)}\right)2^{k(\eta)}.
\end{equation}
Here $k(\eta)$  denotes the number of connected components (or open clusters) of the graph having for vertices $V$ and for edges the open edges in the configuration $\eta$
(recall that an edge $e$ is said to be
open in $\eta$ if $\eta(e)=1$ and closed if $\eta(e)=0$).  Naturally, $Z_{RC}$ is the partition function that makes $
\phi_{p,2}$ a probability measure on $\{0,1\}^E$. The second marginal measure
is given by
\begin{equation}\label{transf}
\forall\,\, \sigma\in \{-1,+1\}^V\qquad
\sum_{\eta \in \{0,1\}^E}{IP}(\eta,\sigma)=
\frac{1}{Z} (1-p)^{|\{\,e\in E:\delta_\sigma(e)=0\,\}|}\,.
\end{equation}
Notice that
$$|\{\,e\in E:\delta_\sigma(e)=0\,\}|\,=\,
\sum_{e\in E}\big(1-\delta_{\sigma}(e)\big)\,=\,
|E|-
\sum_{e\in E}\delta_{\sigma}(e)
\,.$$
Let $\beta>0$ be such that $p=1-e^{-\beta}$. Rewriting the formula~\eqref{transf} with these notations,
we see that the second marginal is
the measure $\mu_{\beta}$ on $\{-1,+1\}^V$
given by
\begin{equation}\label{mubeta}
\forall\,\, \sigma\in \{-1,+1\}^V\qquad \mu_{\beta}(\sigma)
\,=\,
\frac{1}{Z_{I}} \exp\left({\beta} \sum_{e\in E}\delta_{\sigma}(e)\right)  ,\,\,\,
\end{equation}
where  the partition function $Z_I$ is
$$Z_I\,=\,Z\,\exp(\beta|E|)
\,.$$
Let us rewrite the Hamiltonian in a more classical way. 
For an edge $e$ with endpoints $x,y$, we have
\begin{equation}\label{reta}
\delta_\sigma(e)\,=\,\frac{1}{2}\big(1+\sigma(x)\sigma(y)\big)\,.
\end{equation}
Summing over $e\in E$, we get
$$
\sum_{e\in E}\delta_{\sigma}(e)
\,=\,
\sum_{\genfrac{}{}{0pt}{1}{\scriptstyle x,y\in V}{\scriptstyle \{x,y\}\in E}}
\frac{1}{2}\big(1+\sigma(x)\sigma(y)\big)
\,=\,
\frac{1}{2}|E|+
\frac{1}{2}
\sum_{\genfrac{}{}{0pt}{1}{\scriptstyle x,y\in V}{\scriptstyle \{x,y\}\in E}}
\sigma(x)\sigma(y)
\,.$$
The standard Hamiltonian of the Ising model is defined as
\begin{equation}\label{hcbeta}
\forall \sigma\in \{-1,+1\}^V\qquad H(\sigma)\,=\,
-\frac{1}{2}
\sum_{\genfrac{}{}{0pt}{1}{\scriptstyle x,y\in V}{\scriptstyle \{x,y\}\in E}}
\sigma(x)\sigma(y)\,.
\end{equation}
In the end, we have
\begin{equation}\label{mucbeta}
\forall\,\, \sigma\in \{-1,+1\}^V\qquad \mu_{\beta}(\sigma)
\,=\,
	\frac{1}{Z_{\beta}} \exp\left(-{\beta} H(\sigma)\right)  ,\,\,\,
\end{equation}
where  the partition function $Z_\beta$ is given by
$$Z_\beta
\,=\,
\sum_{\sigma\in \{-1,+1\}^V} \exp\left(-{\beta} H(\sigma)\right)  
\,=\,Z_I\,\exp\Big(-\frac{1}{2}\beta|E|\Big)
\,.$$
We refer to the book of Grimmett \cite{Grimmett}
and the references therein for more details about the coupling measure $IP$ between
random-cluster and Ising measures, its history and usefulness.
We refer to the paper of Schonmann \cite{Sch} for a nice presentation of the Gibbs measure of 
the Ising model and the associated dynamics.
\section{The stochastic Ising model}
\label{glauber}
In this section, we define Glauber type dynamics for the Ising model. Each of these dynamics defines
a Markov process $(\sigma_t)_{t\geq 0}$ on the spin configurations which is called a
stochastic Ising model.
A Markov process $(\sigma_t)_{t\geq 0}$ on the spin configurations 
is classically defined through its infinitesimal generator $L$ (see \cite{liggett}).
This generator $L$ acts on  functions $f$ of the spin configuration and it is of the following form. For any function $f:\{-1,+1\}^V\rightarrow \BBr$, we have
$$
\forall\,\, \sigma\in \{-1,+1\}^V\qquad Lf(\sigma)=
\sum_{\sigma'\in\{-1,+1\}^V}c(\sigma,\sigma')(f(\sigma')-f(\sigma))\,.
$$
The quantity $c(\sigma,\sigma')$ is
the rate at which the configuration $\sigma$ is transformed into $\sigma'$
when the system is in the state $\sigma$.
We are usually interested in local dynamics which modify at most one spin at a time. 
We call this type of dynamics Glauber type dynamics.
For a configuration
$\sigma$ and a vertex $x$, we denote by
$\sigma^x$ the configuration obtained from $\sigma$ by flipping the spin at the site $x$.
So we require that $c(\sigma,\sigma')$ vanishes when $\sigma$ and $\sigma'$ differ in more than one spin
and
the quantity $c(\sigma,\sigma^x)$, sometimes denoted by $c(x,\sigma)$,
is
the rate at which the spin at the site $x$ flips
when the system is in the state $\sigma$.
The transition rates $c(\cdot,\cdot)$ are defined on $\{-1,1\}^V\times \{-1,1\}^V$ and 
have thus the following properties:
for any
$\sigma\in \{-1,+1\}^V$ and $x\in V$,
\begin{eqnarray}\label{tauxglauber}
&& {c}(\sigma,\sigma') = 0\,\,\,{\mbox{if}}\,\,\, |\{x\in V,\,\sigma(x)\neq \sigma'(x)\}|\geq 2, {\nonumber}\\
&& {c}(\sigma,\sigma^x)>0\,\,{\mbox{for any}}\,\, x\in V, \\
&& c(\sigma,\sigma)=-\sum_{y\in V}c(\sigma,\sigma^y). {\nonumber}
\end{eqnarray}
We say that the transition
rates $c(\cdot,\cdot)$
satisfy the detailed balance condition with respect to
the Ising measure $\mu_{\beta}$ if
\begin{equation}\label{balanceising}
\forall\,\sigma\in\{-1,1\}^V\quad \forall\,x\in V\qquad
  \mu_{\beta}(\sigma)c(\sigma,\sigma^x)=\mu_{\beta}(\sigma^x)c(\sigma^x,\sigma)\,.
\end{equation}
When this
detailed balance condition
holds, the associated dynamics
is reversible with respect to the Ising measure $\mu_{\beta}$.
Several rates satisfy the conditions~(\ref{tauxglauber}) and the detailed balance condition (\ref{balanceising}).
These dynamics and their fundamental properties are presented in 
the paper of Schonmann \cite{Sch}. In the same paper appears the graphical construction of these
dynamics, which provides a natural and intuitive picture of the associated Markov processes.
Let us present some classical choices. 
Recalling the definition of the Ising Hamiltonian~\eqref{hcbeta},
we define further, for
$\sigma\in \{-1,+1\}^V$ and $x\in V$,
\begin{equation}\label{delta}
	\Delta_xH(\sigma)\,=\, H(\sigma^x)- H(\sigma)\,.
\end{equation}
%
For dynamics which modify at most one spin at a time, we set
	$$c(x,\sigma)=
	c(\sigma,\sigma^x)\,.$$
Here are some possible choices for the rates:

\noindent
{\bf The Metropolis dynamics:}
$$c(x,\sigma)\,=\,
\exp\Big(-\beta\max\big(\Delta_xH(\sigma),0\big)\Big)\,.$$

\noindent
{\bf The Heat bath dynamics:}
$$c(x,\sigma)\,=\,
\frac{1}{1+
\exp\Big(\beta\max\big(\Delta_xH(\sigma),0)\big)\Big)}\,.$$

\noindent
{\bf Unnamed dynamics:}
$$c(x,\sigma)\,=\,
\exp\Big(-\frac{\beta}{2}\Delta_xH(\sigma)\Big)\,.$$
These three choices satisfy 
the detailed balance condition~\eqref{balanceising} and correspond to one spin flip dynamics.
We call them Glauber type dynamics. 
Let us try to express the rates with the help of the functions $\delta_\sigma$ on the edges.
Using the identity~\eqref{reta},
we have, for any $x\in V$,
$$
\Delta_xH(\sigma)\,=\,
\sum_{y\in V: \{x,y\}\in E}
\sigma(x)\sigma(y)
\,=\,\sum_{e\in E_x}\big(2\delta_{\sigma}(e)-1\big)
\,
$$
where the set $E_x$ 
is the set of the edges $e\in E$ having $x$ as endvertex, i.e.,
$$
E_x\,=\,\big\{\,e=\langle x,y\rangle\,\in E:\, y\in V\,\big\}\,.
$$
For the unnamed dynamics, we obtain
\begin{equation*}
	c(x,\sigma)\,=\,
	\exp\left(- \beta \sum_{e\in E_x}\delta_{\sigma}(e)
+\frac{\beta}{2}|E_x|
	\right)
	\,.
\end{equation*}
On the lattice $\mathbb{Z}^d$, we have $|E_x|=2d$ for any $x\in\mathbb{Z}^d$. 
Up to a constant multiplicative factor, the previous rates are equal to
\begin{equation}\label{GD}
	c(x,\sigma)\,=\,
	\exp\left(- \beta \sum_{e\in E_x}\delta_{\sigma}(e)
	\right)
	\,.
\end{equation}
We will mostly use this choice when trying to build a dynamical coupling between the Ising
model and the FK model.
\section{FK dynamics}
\label{fkdyna}
In this section, we define the counterpart of a Glauber type dynamics for the FK model.
We are interested in dynamics which modify at most one edge at a time.
We build a Markov process on percolation configurations, which
is defined through its infinitesimal generator $L$ (see \cite{liggett}).
This generator $L$ acts on  functions $g$ of the percolation configuration and it is of the following form.
For any function $g: \{0,1\}^E \rightarrow \BBr$, we have
$$\forall \eta\in \{0,1\}^E \qquad
Lg(\eta)\,=\, \sum_{e\in E}
c(\eta,\eta^e)(g(\eta^e)-g(\eta))\,,$$
where $\eta^e$ is the configuration obtained from $\eta$ by changing the state of the edge~$e$
	and
the quantity $c(\eta,\eta^e)$ is
the rate at which the edge $e$ changes its state
when the system is in the state $\eta$.
The transition rates $c(\cdot,\cdot)$ are defined on
$\{0,1\}^E \times \{0,1\}^E$
	and have the following properties:
for any
$\eta\in \{0,1\}^E$ and $e\in E$,
\begin{eqnarray}\label{tauxpglauber}
&& {c}(\eta,\eta') = 0\,\,\,{\mbox{if}}\,\,\, |\{e\in E,\,\eta(x)\neq \eta'(e)\}|\geq 2, {\nonumber}\\
&& {c}(\eta,\eta^e)>0\,\,{\mbox{for any}}\,\, e\in E, \\
&& c(\eta,\eta)=-\sum_{e\in E}c(\eta,\eta^e). {\nonumber}
\end{eqnarray}
We say that the transition
rates $c(\cdot,\cdot)$
satisfy the detailed balance condition with respect to
	the FK measure $\phi_{p,2}$ if
\begin{equation}\label{balancefk}
\forall\eta\in \{0,1\}^E
\qquad \forall\,e\in E
	\qquad \phi_{p,2}(\eta) c(\eta,\eta^e)\,=\,
	\phi_{p,2}(\eta^e) c(\eta^e,\eta)\,.
\end{equation}
When this
detailed balance condition
holds, the associated dynamics
is reversible with respect to
	the FK measure $\phi_{p,2}$.
Several rates satisfy the conditions~(\ref{tauxpglauber}) and the detailed balance condition (\ref{balancefk}).
A natural choice is the following.
For
an edge $e=\langle x,y\rangle\in E$ and a configuration
$\eta\in \{0,1\}^E$, we set  $\gamma_{\eta}(e)=1$ if the endpoints $x,y$ of $e$ are connected
	by a path of open edges in $\eta$ which does not use the edge $e$ itself, and we
set  $\gamma_{\eta}(e)=0$ otherwise.
We define then
\begin{equation*}
c(e,\eta):=c(\eta,\eta^e)
\,=\,
\begin{cases}
	1-p&\text{ if }\eta(e)=1\,,\cr
	p&\text{ if }\eta(e)=0\text{ and }\gamma_\eta(e)=1\,,\cr
	p/2&\text{ if }\eta(e)=0\text{ and }\gamma_\eta(e)=0\,.\cr
\end{cases}
\end{equation*}
One can check that these rates satisfy~\eqref{tauxpglauber}
and~
\eqref{balancefk}.
\section{Our dream process}
\label{goal}
The main purpose of this paper is  to
construct a Markov process
on the space $\{0,1\}^E\times \{-1,+1\}^V$
whose marginal on the spins is a Glauber dynamics
and which is
reversible with respect to the coupling measure ${IP}$.
Suppose that
$(\eta_t,\sigma_t)_{t\geq 0}$
is such a Markov process and let
$q((\eta,\sigma),(\eta',\sigma'))$ be its transition rates. These transition rates satisfy the following
usual conditions:
 \begin{equation}\label{qmatrix1}
(\eta,\sigma)\neq (\eta',\sigma')\quad \Longrightarrow  \quad   q((\eta,\sigma), (\eta',\sigma'))\geq 0,
\end{equation}
 and for any $(\eta,\sigma)\in \{0,1\}^E\times \{-1,+1\}^V$,
\begin{equation}\label{qmatrix}
\sum_{(\eta',\sigma')\in \{0,1\}^E\times \{-1,+1\}^V}q((\eta,\sigma), (\eta',\sigma'))=0\,.
\end{equation}
The reversibility property with respect to the coupling measure ${IP}$ is equivalent to
the following detailed balance equation: 
\begin{multline}\label{eqreversible}
\forall (\eta,\sigma),(\eta',\sigma')\in\{0,1\}^E\times \{-1,1\}^V\qquad\hfill\cr
IP(\eta,\sigma)\,q((\eta,\sigma), (\eta',\sigma'))= IP(\eta',\sigma')\,q((\eta',\sigma'), (\eta,\sigma)).
\end{multline}
In general, the marginals of a Markov process are not themselves Markov processes.
Now, we know from Ball and Yeo \cite{Ball} (see theorem \ref{theoball} of the appendix) that the second marginal
$(\sigma_t)_{t\geq 0}$ of the dynamics
is a  Markov jump process with rates ${c}(\sigma,\sigma')$ if and only if
for any $\eta \in \{0,1\}^E$, $\sigma,\sigma' \in \{-1,1\}^V$,
the sum $\sum_{\eta'\in \{0,1\}^E} q((\eta,\sigma),(\eta',\sigma'))$ does not depend on $\eta$.
We have then
\begin{equation}\label{secondmargin}
\forall\eta \in \{0,1\}^E\quad\forall\sigma,\sigma' \in \{-1,1\}^V\qquad
 \sum_{\eta'\in \{0,1\}^E} q((\eta,\sigma),(\eta',\sigma'))= {c}(\sigma,\sigma')\,.
\end{equation}
Suppose that this is the case.
The second marginal $(\sigma_t)_{t\geq 0}$ will then be the Glauber dynamics with the transition rates (\ref{GD}) as soon as the sum in (\ref{secondmargin})
is equal, up to some positive constant, to the transitions rates given
in (\ref{GD}). More precisely, we should have, up to some positive multiplicative constant,
\begin{equation}\label{marginal}
{c}(\sigma,\sigma')
=
\left\{
\begin{array}{cc}
	(1-p)^{|\{e\in E_x,\,\, \delta_{\sigma}(e)=1\}|}\,\, & \text{if} 
	\,\, \sigma'=\sigma^x \,,\\
-\sum_{x\in V}(1-p)^{|\{e\in E_x,\,\, \delta_{\sigma}(e)=1\}|}\,\, & 
	\text{if} \,\, \sigma'=\sigma \,,\\
	0 & \text{otherwise,}
\end{array}
\right.
\end{equation}
where $E_x$ denotes the set of the edges $e$ in $E$ having $x$ as endvertex.
Ideally, we would also wish that
the first marginal $(\eta_t)_{t\geq 0}$ is a Markov jump process with rates ${\tilde c}(\eta,\eta')$.
This will be the case if and only if, for any $\eta \in \{0,1\}^E$, $\eta' \in \{0,1\}^E$, $\sigma \in \{-1,1\}^V$, the sum $\sum_{\sigma'\in \{-1,1\}^V} q((\eta,\sigma),(\eta',\sigma'))$ does not depend on $\sigma$.
We would have then
\begin{equation}\label{premieremargin}
\forall\eta, \eta' \in \{0,1\}^E\quad \forall\sigma \in \{-1,1\}^V\qquad
  \sum_{\sigma'\in \{-1,1\}^V} q((\eta,\sigma),(\eta',\sigma'))={\tilde c}(\eta,\eta')\,.
\end{equation}
So, to sum up,
we dream of constructing a Markov process on $\{0,1\}^E\times \{-1,1\}^V$ with transition rates satisfying (\ref{qmatrix1}), (\ref{qmatrix}), (\ref{eqreversible}), (\ref{secondmargin}), (\ref{marginal})
and~(\ref{premieremargin}),
that is a Markov jump process $(\eta_t,\sigma_t)_{t\geq 0}$ which is reversible with respect to the $IP$ measure,
whose marginal on
the spins is a Glauber dynamics, and whose marginal on the edges is Markovian.
\section{One change at a time}
\label{one}

We present here our first try.
We build a simple dynamics which updates at most
one site or one edge at each time, which is reversible with respect to the coupling measure $IP$
and whose marginal on the spins has the same rates as a Glauber dynamics, although it is not Markovian.
Recall that, for $x\in V$,
$E_x$ is the set of the edges $e$ in $E$ having $x$ as endvertex, i.e.,
$$
E_x\,=\,\big\{\,e=\langle x,y\rangle\,\in E:\, y\in V\,\big\}\,.
$$
For $\sigma\in \{-1,1\}^V$ and $x\in V$, $\sigma^x$ is the element of $\{-1,1\}^V$ obtained from $\sigma$ by reversing
the spin at $x$, i.e.,
$$
\sigma^x(x)=-\sigma(x)\,,\qquad
\forall\, y\in V\setminus \{x\}\quad \sigma^x(y)=\sigma(y)\,.
$$
For $\eta\in \{0,1\}^E$ and $e\in E$,
we denote by
$\eta^e$ the element of $\{0,1\}^E$ obtained from $\eta$ by changing the value of $\eta(e)$, i.e.,
$$\eta^e(e)=1-\eta(e)\,,\qquad
\forall f\in E\setminus\{\,e\,\}\quad
\eta'(f)=\eta(f)\,.$$
Before introducing the transition rates of the one change dynamics, we give a condition
that the transition rates have to fulfill in order to have the properties
announced at the beginning of the section.
\begin{lem}\label{lemn}
Let $(\eta_t,\sigma_t)_{t\geq 0}$ be a  Markov jump process, which updates at most
one site or one edge at each time  and which is reversible with respect to the coupling measure IP.
Let us denote by
$q((\eta,\sigma), (\eta',\sigma'))$ the transition rates of 
$(\eta_t,\sigma_t)_{t\geq 0}$.
Let $(\eta,\sigma)$ be a pair of compatible configurations.
If there exists $x\in V$ such that
$q((\eta,\sigma), (\eta,\sigma^x))\neq 0$,
then all the edges of $E_x$ are closed in~$\eta$.
\end{lem}
\begin{proof}
The detailed balance equation (\ref{eqreversible}) yields that,
for any $\eta\in \{0,1\}$, $\sigma\in \{-1,1\}^V$ and any $x\in V$,
$$
IP(\eta,\sigma)q((\eta,\sigma),(\eta,\sigma^x))= IP(\eta,\sigma^x)q((\eta,\sigma^x),(\eta,\sigma)),
$$
which gives, thanks to (\ref{ipsimply}),
$$
\BBone_{(\eta,\sigma)\in {\cal C}}\,q((\eta,\sigma),(\eta,\sigma^x))= \BBone_{(\eta,\sigma^x)\in {\cal C}}\, q((\eta,\sigma^x),(\eta,\sigma)).
$$
Since $(\eta,\sigma)$ are supposed to be compatible, the last equality implies that
$$
q((\eta,\sigma),(\eta,\sigma^x))= q((\eta,\sigma^x),(\eta,\sigma))\prod_{e\in E_x}\BBone_{\eta(e)\leq 1-\delta_{\sigma}(e)}\,.
$$
Consequently, if $
q((\eta,\sigma),(\eta,\sigma^x))\neq 0$ then necessarily $ \eta(e)\leq 1-\delta_{\sigma}(e)$ for any $e\in E_x$.
Since  the configurations $\eta$ and $\sigma$ are compatible, necessarily
$\eta(e)= 0$ for any $e\in E_x$.
\end{proof}
\noindent We define next the transition rates of our dynamics.
\begin{defi}\label{def0}
Let $(\eta,\sigma)$, $(\eta',\sigma')$ be  two elements of $\{0,1\}^E\times\{-1,1\}^V$.
We consider several cases:

\noindent
$\bullet$ If $\eta'=\eta$ and there exists $x\in V$ such that
$\sigma'=\sigma^x$, $\eta(e)=0$ for any $e\in E_x$,
then we define $ c((\eta,\sigma), (\eta',\sigma'))= 1\,.  $

\noindent
$\bullet$ If $\sigma'=\sigma$ and there exists $e\in E$ such that $\eta'=\eta^e$,
then we set
$$c((\eta,\sigma), (\eta',\sigma'))= p\BBone_{\eta(e)=0}\delta_{\sigma}(e)+ (1-p)\BBone_{\eta(e)=1} \,.$$

\noindent
$\bullet$ Otherwise, if
$(\eta,\sigma)\neq (\eta',\sigma')$, then we set
$c((\eta,\sigma), (\eta',\sigma'))= 0$.

\noindent
$\bullet$ Finally, if $(\eta,\sigma)= (\eta',\sigma')$, then we set
$$
c((\eta,\sigma), (\eta,\sigma))= -\sum_{(\eta',\sigma'),\, (\eta',\sigma')\neq (\eta,\sigma)}c((\eta,\sigma), (\eta',\sigma'))\,.
$$
\end{defi}

\begin{pro}\label{prorev1}
The transition rates introduced in definition \ref{def0}  satisfy the   detailed balance equation  (\ref{eqreversible}).
\end{pro}
\begin{proof}
We check first that,
for any $(\eta,\sigma)$  in $\{0,1\}^E\times \{-1,1\}^V$ and $e\in E$,
\begin{equation}\label{t2}
IP(\eta,\sigma)c((\eta,\sigma), (\eta^e,\sigma))= IP(\eta^e,\sigma)c((\eta^e,\sigma), (\eta,\sigma))
\end{equation}
Let $(\eta,\sigma)$  in $\{0,1\}^E\times \{-1,1\}^V$ and $e\in E$ be fixed.
We have, since $\eta(e)=1-\eta^e(e)$,
\begin{eqnarray}\label{t1}
&& \left(p\BBone_{\eta(e)=1}\delta_{\sigma}(e)+ (1-p)\BBone_{\eta(e)=0} \right)\times\left(p\BBone_{\eta(e)=0}\delta_{\sigma}(e)+ (1-p)\BBone_{\eta(e)=1} \right){\nonumber}\\
&&= \left(p\BBone_{\eta^e(e)=1}\delta_{\sigma}(e)+ (1-p)\BBone_{\eta^e(e)=0} \right)\times\left(p\BBone_{\eta^e(e)=0}\delta_{\sigma}(e)+ (1-p)\BBone_{\eta^e(e)=1} \right).
\end{eqnarray}
We also have, since $\eta^e=\eta$ on $E\setminus\{e\}$,
\begin{eqnarray}\label{nrev}
&& \prod_{f\in E\setminus\{e\}}\left(p\BBone_{\eta(f)=1}\delta_{\sigma}(f)+ (1-p)\BBone_{\eta(f)=0} \right){\nonumber}\\
&&= \prod_{f\in E\setminus\{e\}}\left(p\BBone_{\eta^e(f)=1}\delta_{\sigma}(f)+ (1-p)\BBone_{\eta^e(f)=0} \right).
\end{eqnarray}
Noting that
$$
IP(\eta,\sigma)=\frac{1}{Z}\left(p\BBone_{\eta(e)=1}\delta_{\sigma}(e)+ (1-p)\BBone_{\eta(e)=0} \right)\prod_{f\in E\setminus\{e\}}\left(p\BBone_{\eta(f)=1}\delta_{\sigma}(f)+ (1-p)\BBone_{\eta(f)=0} \right),
$$
we get (\ref{t2}) by multiplying each side of~\eqref{nrev} by~\eqref{t1}.
We now prove that, for any $(\eta,\sigma)$  in $\{0,1\}^E\times \{-1,1\}^V$  and any $x\in V$,
\begin{equation}\label{t3}
IP(\eta,\sigma)c((\eta,\sigma), (\eta,\sigma^x))= IP(\eta,\sigma^x)c((\eta,\sigma^x), (\eta,\sigma)).
\end{equation}
Noting that $$c((\eta,\sigma), (\eta,\sigma^x))= \prod_{e\in E_x}\BBone_{\eta(e)=0},$$
and that $\delta_{\sigma}(e)=\delta_{\sigma^x}(e)$ for any $e\in E\setminus E_x$, we deduce that
\begin{eqnarray*}\label{e1}
&&c((\eta,\sigma), (\eta,\sigma^x))IP(\eta,\sigma)= {\nonumber}\\
&&= \frac{1}{Z}\left(\prod_{e\in E_x}\BBone_{\eta(e)=0}\right)\left(\prod_{e\in E}\left(p\BBone_{\eta(e)=1}\delta_{\sigma}(e) + (1-p)\BBone_{\eta(e)=0}\right)\right){\nonumber}\\
&&=\frac{1}{Z}\left(\prod_{e\in E_x}((1-p)\BBone_{\eta(e)=0})\right)\left(\prod_{e\in E\setminus E_x}\left(p\BBone_{\eta(e)=1}\delta_{\sigma}(e) + (1-p)\BBone_{\eta(e)=0}\right)\right)
{\nonumber}\\
&&= \frac{1}{Z}\left(\prod_{e\in E_x}((1-p)\BBone_{\eta(e)=0})\right)\left(\prod_{e\in E\setminus E_x}\left(p\BBone_{\eta(e)=1}\delta_{\sigma^x}(e) + (1-p)\BBone_{\eta(e)=0}\right)\right)
{\nonumber}
\\
&&= \frac{1}{Z}\left(\prod_{e\in E_x}\BBone_{\eta(e)=0}\right)\left(\prod_{e\in E}\left(p\BBone_{\eta(e)=1}\delta_{\sigma^x}(e) + (1-p)\BBone_{\eta(e)=0}\right)\right)
{\nonumber}\\
&&= c((\eta,\sigma^x), (\eta,\sigma))IP(\eta,\sigma^x).
\end{eqnarray*}
This finishes the proof of
condition~(\ref{t3}).
	Since the transition rates $c((\eta,\sigma), (\eta',\sigma'))$ given in definition~\ref{def0}
	allow to perform at most one change at a time,
	the reversibility condition reduces to the two conditions~(\ref{t2}) and~(\ref{t3}),
so the proof of Proposition \ref{prorev1} is complete.
\end{proof}
\noindent The following proposition shows that,
if the initial condition $(\eta_0,\sigma_0)$ is distributed as $IP$,
then
the infinitesimal behavior of the marginal on the spins of the process $(\eta_t,\sigma_t)_{t\geq 0}$,
whose rates are introduced in definition \ref{def0}, is
the same as the transition rates given in (\ref{marginal}).
\begin{pro}\label{marg}
Let $(\eta_t,\sigma_t)_{t\geq 0}$ be a  Markov jump process with the transition rates
introduced in definition \ref{def0}. Suppose that  the pair $(\eta_0,\sigma_0)$ is distributed according to $IP$. Then, for any $x\in V$,
for any $\sigma\in \{-1,1\}^V$ and any $s\geq 0$,
\begin{equation}\label{limite}
\lim_{t\rightarrow 0}\frac{1}{t}\BBp(\sigma_{t+s}=\sigma^x|\sigma_s=\sigma)=
(1-p)^{|\{e\in E_x,\,\delta_{\sigma}(e)=1\}|}.
\end{equation}
\end{pro}
\begin{proof} We have, for any $\sigma\in \{-1,1\}^V$, $x\in V$ and $s\geq 0$,
\begin{eqnarray}\label{ee1}
&& \BBp(\sigma_{t+s}=\sigma^x|\sigma_s=\sigma) = \sum_{\eta\in \{0,1\}^E}\sum_{\eta'\in \{0,1\}^E}\BBp(\sigma_{t+s}=\sigma^x, \eta_{t+s}=\eta',\eta_s=\eta|\sigma_s=\sigma){\nonumber}\\
&&= \sum_{\eta\in \{0,1\}^E}\sum_{\eta'\in \{0,1\}^E}\BBp(\sigma_{t+s}=\sigma^x, \eta_{t+s}=\eta'|\sigma_s=\sigma, \eta_s=\eta) \BBp( \eta_s=\eta|\sigma_s=\sigma).
\end{eqnarray}
Since $(\eta_t,\sigma_t)_{t\geq 0}$ is a  Markov jump process with the transition rates $c((\eta,\sigma),(\eta',\sigma'))$
introduced in definition \ref{def0}, then
\begin{eqnarray}\label{ee2}
\lim_{t\rightarrow 0}\frac{1}{t}\BBp(\sigma_{t+s}=\sigma^x, \eta_{t+s}=\eta'|\sigma_s=\sigma, \eta_s=\eta)
= c((\eta,\sigma),(\eta',\sigma^x))=\BBone_{\eta'=\eta}\prod_{e\in E_x}\BBone_{\eta(e)=0}.
\end{eqnarray}
Hence we get, combining (\ref{ee1}) and (\ref{ee2}), for any $s\geq 0$,
\begin{eqnarray*}
&& \lim_{t\rightarrow 0}\frac{1}{t}\BBp(\sigma_{t+s}=\sigma^x|\sigma_s=\sigma)=\sum_{\eta\in \{0,1\}^E}\left(\prod_{e\in E_x}\BBone_{\eta(e)=0}\right)\BBp( \eta_s=\eta|\sigma_s=\sigma).
\end{eqnarray*}
The measure $IP$ is, by proposition \ref{prorev1}, a reversible  and a stationary measure for the process $(\eta_t,\sigma_t)_{t\geq 0}$.
 Hence, if $(\eta_0,\sigma_0)$ is distributed as $IP$, then, for any $s\geq 0$,  $(\eta_s,\sigma_s)$ is also distributed as $IP$. Consequently, under the hypothesis of proposition \ref{marg}, we have
$$
\BBp( \eta_s=\eta|\sigma_s=\sigma)=\frac{\BBp( \eta_s=\eta,\sigma_s=\sigma)}{\BBp(\sigma_s=\sigma)}=\frac{IP(\eta,\sigma)}
{\sum_{\eta'\in \{0,1\}^E}IP(\eta',\sigma)}\,.$$
We conclude that
\begin{equation}\label{sumtaux}
\lim_{t\rightarrow 0}\frac{1}{t}\BBp(\sigma_{t+s}=\sigma^x|\sigma_s=\sigma)
\, = \,
\frac{\sum_{\eta\in \{0,1\}^E,\,\, \eta\equiv 0\, on\, E_x}IP(\eta,\sigma)}
{\sum_{\eta'\in \{0,1\}^E}IP(\eta',\sigma)}
\,.
\end{equation}
We compute these sums, starting with the very definition of the coupling measure $IP$. We have
\begin{equation}\label{sumaux}
\sum_{\eta\in \{0,1\}^E,\,\, \eta\equiv 0\, on\, E_x}IP(\eta,\sigma)
\,=\,
\frac{1}{Z} (1-p)^{|\{\,e\in E:\delta_\sigma(e)=0\,\}|}
 (1-p)^{|\{\,e\in E_x:\delta_\sigma(e)=1\,\}|}
\,.
\end{equation}
The sum in the denominator has already been computed in equation~\eqref{transf}.
Equations~\eqref{transf}, \eqref{sumtaux} and~\eqref{sumaux} together yield the desired result~\eqref{limite}.
\end{proof}
We consider next the marginal on the edges and we obtain a similar result.
%
\begin{pro}\label{marv}
Let $(\eta_t,\sigma_t)_{t\geq 0}$ be a  Markov jump process with the transition rates
introduced in definition \ref{def0}.
Suppose that  the pair $(\eta_0,\sigma_0)$ is distributed according to $IP$. Then, for any $e\in E$,
for any $\eta\in \{0,1\}^E$ and any $s\geq 0$,
\begin{equation}\label{gimite}
\lim_{t\rightarrow 0}\frac{1}{t}\BBp(\eta_{t+s}=\eta^e|\eta_s=\eta)=
	(1-p)\BBone_{\eta(e)=1}+
	p\BBone_{\eta(e)=0,\, \gamma_\eta(e)=1 }
	+ \frac{p}{2}\BBone_{\eta(e)=0,\, \gamma_\eta(e)=0 }\,.
\end{equation}
\end{pro}
\begin{proof} We have, for any $\eta\in \{0,1\}^E$, $e\in E$ and $s\geq 0$,
\begin{eqnarray}\label{ff1}
&&
\BBp(\eta_{t+s}=\eta^e|\eta_s=\eta)
	= \sum_{\sigma\in \{-1,1\}^V}\sum_{\sigma'\in \{-1,1\}^V}
	\BBp(\sigma_{t+s}=\sigma', \eta_{t+s}=\eta^e,
	\sigma_s=\sigma | \eta_s=\eta {\nonumber})\\
	&&= \sum_{\sigma\in \{-1,1\}^V}\sum_{\sigma'\in \{-1,1\}^V}
	\BBp(\sigma_{t+s}=\sigma', \eta_{t+s}=\eta^e|\sigma_s=\sigma, \eta_s=\eta)
	\BBp( \sigma_s=\sigma | \eta_s=\eta).
\end{eqnarray}
Since $(\eta_t,\sigma_t)_{t\geq 0}$ is a  Markov jump process with the transition rates $c((\eta,\sigma),(\eta',\sigma'))$
introduced in definition \ref{def0}, then
 \begin{multline}\label{ff2}
\lim_{t\rightarrow 0}\frac{1}{t}
	\BBp(\sigma_{t+s}=\sigma', \eta_{t+s}=\eta^e|\sigma_s=\sigma, \eta_s=\eta)
= c((\eta,\sigma),(\eta^e,\sigma'))\hfill\cr
	=\BBone_{\sigma'=\sigma}
	\big(p\BBone_{\eta(e)=0}\delta_{\sigma}(e)+ (1-p)\BBone_{\eta(e)=1}\big) \,.
\end{multline}
Combining (\ref{ff1}) and (\ref{ff2}), 
we get, for any $s\geq 0$,
\begin{multline*}
 \lim_{t\rightarrow 0}\frac{1}{t}
\BBp(\eta_{t+s}=\eta^e|\eta_s=\eta)
	=\hfill\cr
	\sum_{\sigma\in \{-1,1\}^V}
	\big(p\BBone_{\eta(e)=0}\delta_{\sigma}(e)+ (1-p)\BBone_{\eta(e)=1}\big)
	\BBp( \sigma_s=\sigma | \eta_s=\eta).
\end{multline*}
The measure $IP$ is, by proposition \ref{prorev1}, a reversible  and a stationary measure for the process $(\eta_t,\sigma_t)_{t\geq 0}$.
 Hence, if $(\eta_0,\sigma_0)$ is distributed as $IP$, then, for any $s\geq 0$,  $(\eta_s,\sigma_s)$ is also distributed as $IP$. Consequently, under the hypothesis of proposition \ref{marv}, we have
$$
\BBp( \sigma_s=\sigma | \eta_s=\eta)
	=\frac{\BBp( \eta_s=\eta,\sigma_s=\sigma)}{\BBp(\eta_s=\eta)}=\frac{IP(\eta,\sigma)}
{\sum_{\sigma'\in \{-1,1\}^V}IP(\eta,\sigma')}\,.$$
We conclude that
\begin{equation}\label{sumfaux}
 \lim_{t\rightarrow 0}\frac{1}{t}
\BBp(\eta_{t+s}=\eta^e|\eta_s=\eta)
\, = \,
	\frac{\sum_{\sigma\in \{-1,1\}^V}
	\big(p\BBone_{\eta(e)=0}\delta_{\sigma}(e)+ (1-p)\BBone_{\eta(e)=1}\big)
	IP(\eta,\sigma)}
{\sum_{\sigma'\in \{-1,1\}^V}IP(\eta,\sigma')}
\,.
\end{equation}
We compute these sums, starting with the very definition of the coupling measure $IP$. We have
\begin{multline}\label{tumaux}
	{\sum_{\sigma\in \{-1,1\}^V}
	\big(p\BBone_{\eta(e)=0}\delta_{\sigma}(e)+ (1-p)\BBone_{\eta(e)=1}\big)
	IP(\eta,\sigma)}
\,=\,\hfill\cr
	\Big(
	p\BBone_{\eta(e)=0,\, \gamma_\eta(e)=1 }
	+ \frac{p}{2}\BBone_{\eta(e)=0,\, \gamma_\eta(e)=0}
	+
	(1-p)\BBone_{\eta(e)=1}
	\Big)
	\Phi_{p,2}(\eta)
	\,.
\end{multline}
The sum in the denominator has already been computed in equation~\eqref{tramsf}, it is equal to $\Phi_{p,2}(\eta)$.
Equations~\eqref{tramsf}, \eqref{sumfaux} and~\eqref{tumaux} together yield the desired result~\eqref{gimite}.
\end{proof}
\noindent Although the infinitesimal behavior of the marginal on the spins of $(\eta_t,\sigma_t)_{t\geq 0}$ is the same
as for the Glauber dynamics with the transition rates given in (\ref{marginal}), the process
$(\sigma_t)_{t\geq 0}$ does not   evolve according to a Glauber dynamics, because it is not a Markovian process. 
Similarly, the marginal on the edges is not a Markovian process. We summarize these results in the following theorem.
\begin{theo}\label{theo20}
We suppose that the set of edges $E$ is not empty.
Let $(\eta_t,\sigma_t)_{t\geq 0}$ be a  Markov jump process with the transition rates
introduced in definition \ref{def0}. Then $(\eta_t,\sigma_t)_{t\geq 0}$ is reversible with respect to the coupling measure IP.
However its two marginal processes $(\eta_t)_{t\geq 0}$ and $(\sigma_t)_{t\geq 0}$ are  not Markovian jump processes.
\end{theo}
\begin{proof}
The reversibility property has been proved in proposition  \ref{prorev1}. Yet the sum $$\sum_{\eta'\in \{0,1\}^E}c((\eta,\sigma), (\eta',\sigma^x))$$  depends on $\eta$ for some $\sigma\in \{-1,1\}^V$ and $x\in V$.
In fact, we have, for any  $\sigma\in \{-1,1\}^V$ and $x\in V$,
$$
\sum_{\eta'\in \{0,1\}^E}c((\eta,\sigma), (\eta',\sigma^x))=c((\eta,\sigma), (\eta,\sigma^x))=\prod_{e\in E_x}\BBone_{\eta(e)=0}.
$$
Similarly, the sum $$\sum_{\sigma'\in \{-1,1\}^V}c((\eta,\sigma), (\eta^e,\sigma'))$$
depends on $\sigma$ for some
$\eta\in \{0,1\}^E$ and $e\in E$.
In fact, we have, for any
$\eta\in \{0,1\}^E$ and $e\in E$,
$$
\sum_{\sigma'\in \{-1,1\}^V}c((\eta,\sigma), (\eta^e,\sigma'))= c((\eta,\sigma), (\eta^e,\sigma))= p\BBone_{\eta(e)=0}\delta_{\sigma}(e)+ (1-p)\BBone_{\eta(e)=1}.
$$
The proof of Theorem \ref{theo20} is complete thanks to Theorem 3.1 in Ball and Yeo \cite{Ball} (see theorem \ref{theoball} in the appendix).
\end{proof}

\noindent Theorem \ref{theo20} shows that the Markov jump process with the transition rates given in definition \ref{def0} does not fulfill our dream,
which was to build a Markov jump process on $\{0,1\}^E\times \{-1,1\}^V$, reversible with respect to the $IP$ measure,
with a Glauber dynamics as the marginal on the spins and such that the marginal on the edges is Markovian.  The following theorem shows  that there is no hope to realize this dream with  a Markov jump process $(\eta_t,\sigma_t)_{t\geq 0}$ which updates at most one site or one edge at each time.
\begin{theo}\label{theoonechange}
Let $(\eta_t,\sigma_t)_{t\geq 0}$ be a Markov jump process which updates at most one site or one edge at each time and whose marginal on the spins is a Markov jump process with transition rates  satisfying (\ref{tauxglauber}). Then $(\eta_t,\sigma_t)_{t\geq 0}$ cannot be reversible with respect to the coupling measure $IP$.
\end{theo}

\begin{proof}
Let $(\eta_t,\sigma_t)_{t\geq 0}$ be a Markov jump process satisfying the hypothesis of theorem \ref{theoonechange} and
let $q((\eta,\sigma),(\eta',\sigma'))$ be its transition rates. Since this Markov process
updates at most one site or one edge at each time, then
\begin{equation}
\sum_{\eta'\in \{0,1\}^E} q((\eta,\sigma),(\eta',\sigma^x))= q((\eta,\sigma),(\eta,\sigma^x)).
\end{equation}
Since the  marginal  on the spins of $(\eta_t,\sigma_t)_{t\geq 0}$ is a Markov process with
transition rates $c(\sigma,\sigma')$ satisfying (\ref{tauxglauber}), Theorem 3.1 in Ball and Yeo \cite{Ball} gives that
$$\forall \eta\in \{0,1\}^E \qquad
\sum_{\eta'\in \{0,1\}^E} q((\eta,\sigma),(\eta',\sigma^x)))=c(\sigma,\sigma^x)\,.$$
Hence, for any $\eta\in \{0,1\}^E$, any $\sigma\in \{-1,1\}^V$ and any $x\in V$
$$
q((\eta,\sigma),(\eta,\sigma^x))= c(\sigma,\sigma^x).
$$
Suppose now that the detailed balance equation (\ref{eqreversible}) is satisfied. We deduce then from (\ref{tauxglauber}) and lemma \ref{lemn}
that necessarily $\eta(e)=0$ for $e\in E_x$.
We conclude that the detailed balance equation (\ref{eqreversible}) is not satisfied for every $\eta\in \{0,1\}^E$ and every $\sigma\in \{-1,1\}^V$.  The proof of Theorem \ref{theoonechange} is complete.
\end{proof}

\noindent Theorem~$3.2$ in \cite{Ball} (see theorem \ref{theoball2} in the appendix) implies that, for a process like the above one, in which
at most one marginal process moves at each time step, the marginal processes are Markovian if and only
if they are independent. If that were the case, then the equilibrium measure would be a product
measure. Therefore  it is impossible to realize our dream with a process
which changes only one edge or one spin at a time. In the next section, we construct a process
which can change simultaneously one spin and its incident edges.

\section{One site and the incident edges}
\label{two}
We build here a dynamics which updates at most
one site and its incident edges at each time, which is reversible with respect to the coupling measure $IP$
and whose marginal on the spins is a Glauber dynamics.
For $x\in V$, recall that
$E_x$ is the set of the edges $e$ in $E$ having $x$ as endvertex, i.e.,
$$
E_x=\{e\in E,\, e=\langle x,y\rangle,\, y\in V\},
$$
and define
$$
\cC_x=\big\{(\eta,\sigma)\in \{0,1\}^E\times \{-1,1\}^V: \eta(e) \leq \delta_{\sigma}(e)\,\,\,\text{for any}\,\,\,
e\in E_x\big\}.
$$
The following definition gives the transition rates of the dynamics.
\begin{defi}\label{def1}
Let $(\eta,\sigma)$, $(\eta',\sigma')$ be two elements of $\{0,1\}^E\times\{-1,1\}^V$.
We consider several cases:

\noindent
$\bullet$ If there exists $x\in V$ such that
$\sigma'=\sigma^x$ and $\eta'(e)=\eta(e)$ for any $e\in E\setminus E_x$, then we define
$$
c((\eta,\sigma), (\eta',\sigma'))= (1-p)^{|\{e\in E_x,\,\, \eta'(e)=0\}|}p^{|\{e\in E_x,\,\,\, \eta'(e)=1\}|}\BBone_{(\eta',\sigma')\in \cC_x}\,.
$$

\noindent
$\bullet$ Otherwise, if
$(\eta,\sigma)\neq (\eta',\sigma')$, then we set
$c((\eta,\sigma), (\eta',\sigma'))= 0$.

\noindent
$\bullet$ Finally, if $(\eta,\sigma)= (\eta',\sigma')$, then we set
$$
c((\eta,\sigma), (\eta,\sigma))= -\sum_{(\eta',\sigma'),\, (\eta',\sigma')\neq (\eta,\sigma)}c((\eta,\sigma), (\eta',\sigma'))\,.
$$
\end{defi}
%
\begin{pro}\label{proreversible}
The transition rates introduced in definition \ref{def1}  satisfy the   detailed balance equation  (\ref{eqreversible}).
\end{pro}
\begin{proof}
Let $(\eta,\sigma)$ and $(\eta',\sigma')$ be fixed in $\{0,1\}^E\times \{-1,1\}^V$.
We have to prove (\ref{eqreversible}) only in the case where there exists $x\in V$ for which
\begin{equation}\label{vbn}
\sigma'=\sigma^x\,,\qquad \forall e\in E\setminus E_x\quad \eta'(e)= \eta(e)\,,
\end{equation}
because in
all the other cases,  the detailed balance equation (\ref{eqreversible}) is  trivially satisfied.
Let us fix $x\in V$.
Suppose first that $(\eta,\sigma)$ and $(\eta',\sigma')$ are two elements of $\cC_x$
	satisfying~(\ref{vbn}).
We get then, for $e\in E_x$,
$$\BBone_{\eta(e)=1}\delta_{\sigma}(e)= \BBone_{\eta(e)=1}\,,\qquad
\BBone_{\eta'(e)=1}\delta_{\sigma^x}(e)= \BBone_{\eta'(e)=1}\,,$$
therefore
\begin{eqnarray*}
 &&\kern-23pt\left(\prod_{e\in E_x}\left(p\BBone_{\eta(e)=1}\delta_{\sigma}(e)+ (1-p)\BBone_{\eta(e)=0} \right)\right)c((\eta,\sigma), (\eta',\sigma^x))\\
&&= p^{|\{e\in E_x,\,\,\, \eta(e)=1\}|}(1-p)^{|\{e\in E_x,\,\, \eta(e)=0\}|}(1-p)^{|\{e\in E_x,\,\, \eta'(e)=0\}|}p^{|\{e\in E_x,\,\,\, \eta'(e)=1\}|}\\
&&=\left(\prod_{e\in E_x}\left(p\BBone_{\eta'(e)=1}\delta_{\sigma^x}(e)+ (1-p)\BBone_{\eta'(e)=0} \right)\right)c((\eta',\sigma^x), (\eta,\sigma)).
\end{eqnarray*}
Now we  multiply each member of the last equality by
$$\prod_{e\in E\setminus E_x}\left(p\BBone_{\eta(e)=1}\delta_{\sigma}(e)+ (1-p)\BBone_{\eta(e)=0}\right)
\,=\,
\prod_{e\in E\setminus E_x}\left(p\BBone_{\eta'(e)=1}\delta_{\sigma^x}(e)+ (1-p)\BBone_{\eta'(e)=0}\right)
\,$$
and we obtain the detailed balance equation (\ref{eqreversible}).
Suppose now that $(\eta,\sigma)\in \cC_x$ and $(\eta',\sigma^x)\notin \cC_x$. We have then, by the definition of the transition rates,
$$
IP(\eta,\sigma)c((\eta,\sigma),(\eta',\sigma^x))=0\,.
$$
Moreover
$IP(\eta',\sigma^x)=0$ because
$(\eta',\sigma^x)\notin \cC$,
therefore $IP(\eta',\sigma^x)c((\eta',\sigma^x), (\eta,\sigma))=0$.
So here again equation (\ref{eqreversible}) is satisfied. Finally suppose that $(\eta,\sigma)\notin \cC_x$.
Then $IP(\eta,\sigma)=0$ and so $IP(\eta,\sigma)c((\eta,\sigma),(\eta',\sigma^x))=0$ which is equal to $IP(\eta',\sigma^x)c((\eta',\sigma^x),(\eta,\sigma))$, since  the transition rates
$c((\eta',\sigma^x),(\eta,\sigma))$ vanish as soon as $(\eta,\sigma)\notin \cC_x$.
We conclude that the detailed balance equation (\ref{eqreversible}) is always satisfied.
\end{proof}

\noindent We compute next the marginal dynamics on the spins.

\begin{pro}\label{pro1} For any $\eta$ in $\{0,1\}^E$ and any $\sigma,\sigma'$ in $\{-1,1\}^V$,
the sum
$$\sum_{\eta'\in \{0,1\}^E}c((\eta,\sigma), (\eta',\sigma'))$$
does not depend on $\eta$ and it is equal to
$$
c(\sigma,\sigma')
\,=\,\left\{
\begin{array}{cc}
(1-p)^{|\{e\in E_x,\,\, \delta_{\sigma}(e)=1\}|}\,\, & if \,\, \sigma'=\sigma^x\,, \\
-\sum_{x\in V}(1-p)^{|\{e\in E_x,\,\, \delta_{\sigma}(e)=1\}|}\,\, & if \,\, \sigma'=\sigma\,, \\
0 & otherwise\,.
\end{array}
\right.
$$
\end{pro}
\begin{proof}
Let $\eta\in \{0,1\}^E$ be fixed. If $|\{y\in V, \sigma(y)\neq \sigma'(y)\}|\geq 2$ then, by definition \ref{def1},
\begin{equation}\label{a1}
\sum_{\eta'\in \{0,1\}^E}c((\eta,\sigma), (\eta',\sigma'))=0.
\end{equation}
Suppose now that there exists $x\in V$ such that $\sigma'=\sigma^x$. Then, by definition \ref{def1},
\begin{multline*}
\sum_{\eta'\in \{0,1\}^E}c((\eta,\sigma), (\eta',\sigma'))\,=\,\\
 \sum_{\eta'\in \{0,1\}^E,\,\, \eta'\equiv \eta\, on \, E\setminus E_x}(1-p)^{|\{e\in E_x,\,\, \eta'(e)=0\}|}p^{|\{e\in E_x,\,\,\, \eta'(e)=1\}|}\BBone_{(\eta',\sigma^x)\in \cC_x}.\\
\end{multline*}
Our next task is to calculate the last sum.
To this end, we use the same proof as for (\ref{transf}) (recall that $IP(\eta,\sigma)$ is also given by (\ref{ipsimply})). We get 
%
\begin{equation}\label{a2}
\sum_{\eta'\in \{0,1\}^E}c((\eta,\sigma), (\eta',\sigma^x))= (1-p)^{|\{e\in E_x,\,\, \delta_{\sigma}(e)=1\}|}.
\end{equation}
Finally, we have to calculate $
\sum_{\eta'\in \{0,1\}^E}c((\eta,\sigma), (\eta',\sigma))$. We have, from definition \ref{def1},
\begin{align}\label{a3}
 \sum_{\eta'\in \{0,1\}^E}c((\eta,\sigma),& (\eta',\sigma)) = c((\eta,\sigma), (\eta,\sigma)) {\nonumber} \\
&= -\sum_{(\eta',\sigma'),\, (\eta',\sigma')\neq (\eta,\sigma)}c((\eta,\sigma), (\eta',\sigma')) {\nonumber} \\
&= -\sum_{\eta'\in \{0,1\}^E}\sum_{x\in V}c((\eta,\sigma), (\eta',\sigma^x)){\nonumber} \\
&=-\sum_{x\in V}\sum_{\eta'\in \{0,1\}^E}c((\eta,\sigma), (\eta',\sigma^x)){\nonumber} \\
&=-\sum_{x\in V}(1-p)^{|\{e\in E_x,\,\, \delta_{\sigma}(e)=1\}|}.
\end{align}
Equalities (\ref{a1}), (\ref{a2}) and (\ref{a3}) yield the statement of
proposition \ref{pro1}.
\end{proof}

\noindent We compute next the marginal dynamics on the edges.
\begin{pro}\label{tro2}
 Let $\eta$ and $\eta'$ be two different elements of $\{0,1\}^E$.

\noindent
$\bullet$ If there exists $e=\langle x,y\rangle\in E$ such that $\eta'(e)\neq \eta(e)$ and $\eta'\equiv\eta\,\, {\mbox{on}}\,\, E\setminus \{e\}$, then
$$
\sum_{\sigma'\in \{-1,1\}^V}c((\eta,\sigma), (\eta',\sigma'))=\sum_{z\in \{x,y\}}(1-p)^{|\{e\in E_z,\,\, \eta'(e)=0\}|}p^{|\{e\in E_z,\,\, \eta'(e)=1\}|}\BBone_{(\eta',\sigma^z)\in \cC_z}\,.
$$
\noindent
$\bullet$
If there exists $x\in V$ such that $\eta'\equiv \eta$  on $E\setminus E_x$ and the configurations $\eta,\eta'$ differ in more than one edge, then
$$
\sum_{\sigma'\in \{-1,1\}^V}c((\eta,\sigma), (\eta',\sigma'))=(1-p)^{|\{e\in E_x,\,\, \eta'(e)=0\}|}p^{|\{e\in E_x,\,\, \eta'(e)=1\}|}
\BBone_{(\eta',\sigma^x)\in \cC_x}\,.
$$
\noindent
$\bullet$
If for any $x \in V$, $\eta'\neq \eta\,\, {\mbox{on}}\,\, E\setminus E_x$, then
$$
\sum_{\sigma'\in \{-1,1\}^V}c((\eta,\sigma), (\eta',\sigma'))=0\,.
$$
\end{pro}
\begin{proof}
Let $\eta$ and $\eta'$ be two fixed different elements of $\{0,1\}^E$.
If for any $x\in V$, $\eta$ and $\eta'$ are different on $E\setminus E_x$ then we get, by definition \ref{def1},
$$
\forall\,\, \sigma'\in \{-1,1\}^V\qquad
c((\eta,\sigma),(\eta',\sigma'))=0\,.
$$
Consequently,
\begin{equation}\label{e1tro2}
\sum_{\sigma'\in \{-1,1\}^V}c((\eta,\sigma),(\eta',\sigma'))=0.
\end{equation}
Suppose now that there exists $x\in V$ such that $\eta\equiv \eta'$ on $E\setminus E_x$. Let us define
the set
$${\cal V}\,=\,\big\{\,y\in V,\,\, \eta\equiv \eta'\,\, on\,\, E\setminus E_y\,\big\}\,.$$
We have
\begin{equation}\label{epro4}
\sum_{\sigma'\in \{-1,1\}^V}c((\eta,\sigma),(\eta',\sigma'))= \sum_{y\in {\cal V}}c((\eta,\sigma),(\eta',\sigma^y)).
\end{equation}
Let us set
$${\cal E}\,=\,\big\{f\in E,\,\, \eta(f)\neq \eta'(f)\,\big\}\,.$$
By hypothesis, this set is not empty and it is a subset of $E_x$ because $\eta\neq \eta'$ while $\eta\equiv \eta'$ on $E\setminus E_x$. Suppose first that ${\cal E}$ contains only one element,
 say ${\cal E}=\{e\}$.
Necessarily one endvertex of $e$ is $x$. Let $y$ be the other endvertex of $e$, so that $e=\langle x,y\rangle$.
We have then
${\cal V}=\{x,y\}$. Indeed, if there exists a vertex $z$ in $V$ such that
$$
z\neq x\,\,\,\text{and}\,\, \eta\equiv \eta'\,\,\text{on}\,\, E\setminus E_z\,,
$$
then, since $e\in E_x$ and $\eta(e)\neq \eta'(e)$,
we have also $e\in E_z$, whence $e=\langle x,z\rangle$ and
$z=y$.
Equation (\ref{epro4}) can now be rewritten as
\begin{equation}\label{e2tro2}
\sum_{\sigma'\in \{-1,1\}^V}c((\eta,\sigma),(\eta',\sigma'))= c((\eta,\sigma),(\eta',\sigma^x))+ c((\eta,\sigma),(\eta',\sigma^y))\,.
\end{equation}
Suppose that ${\cal E}$ contains more than one element.
We claim that, in this case, the set ${\cal V}$ is reduced to $\{x\}$. Indeed, let $e_1$ and $e_2$ be two different elements of ${\cal E}$. Let $y\in {\cal V}$. Since
$\eta\equiv \eta'\,\,\text{on}\,\, E\setminus E_y$, necessarily both $e_1$ and $e_2$ belong to $E_y$.
We know also that both $e_1$ and $e_2$ belong to $E_x$.
Consequently $e_1\in E_x\cap E_y$ and $e_2\in E_x\cap E_y$. It follows that $y=x$ since $e_1$ and $e_2$ are different.
From (\ref{epro4}),
we deduce
that,
whenever ${\cal E}$ contains more than one element,
\begin{equation}\label{e3tro2}
\sum_{\sigma'\in \{-1,1\}^V}c((\eta,\sigma),(\eta',\sigma'))= c((\eta,\sigma),(\eta',\sigma^x))\,.
\end{equation}
Definition \ref{def1} and the three cases of equations~(\ref{e1tro2}), (\ref{e2tro2}) and (\ref{e3tro2})
complete the
proof of Proposition \ref{tro2}.
\end{proof}
\noindent From propositions  \ref{proreversible}, \ref{pro1} and \ref{tro2}, we deduce the following theorem.
\begin{theo}\label{theo2}
Let $(\eta_t,\sigma_t)_{t\geq 0}$ be a  Markov jump process with the transition rates
introduced in definition \ref{def1}. Then $(\eta_t,\sigma_t)_{t\geq 0}$ is reversible with respect to the coupling measure IP.
Its second marginal process
 $(\sigma_t)_{t\geq 0}$ is a Markov process evolving according to a Glauber dynamics while its first marginal $(\eta_t)_{t\geq 0}$ is a non-Markovian jump process.
\end{theo}
\begin{proof}
The reversibility property is deduced from Proposition  \ref{proreversible}. Proposition \ref{pro1}
shows that the sum 
$$\sum_{\eta'\in \{0,1\}^E}c((\eta,\sigma), (\eta',\sigma'))$$ 
does not depend on $\eta$ for any $\sigma,\sigma'\in \{-1,1\}^V$,
moreover this sum is equal to the  transition rates given in (\ref{marginal}).
The presence of the indicator function
$\BBone_{(\eta',\sigma^x)\in \cC_x}$ in the formula of
proposition \ref{tro2} shows that the sum $\sum_{\sigma'\in \{-1,1\}^V}c((\eta,\sigma), (\eta',\sigma'))$ depends on $\sigma$ for some $\eta,\eta'\in \{0,1\}^E$.
The proof of Theorem \ref{theo2} is complete thanks to Ball and Yeo \cite{Ball}.
\end{proof}
So this second dynamics does not fulfill our dream. 
\section{The dream is not feasible}\label{general}
A Markov process on $\{0,1\}^E\times \{-1,1\}^V$ is said to be a Markovian coupling if both its  marginals  are Markov processes.
In this section, we will work with the following additional hypothesis on the graph $(V,E)$.
\begin{hypo}\label{adhy}
We suppose that there exists a vertex $\bx\in V$ which belongs to at least
four edges. We suppose that there exists a spin configuration $\bsi$ such that 
$\delta_{\bsi}(e)=0$ for any $e\in  E$.
\end{hypo}
\noindent
For instance, when the graph is bipartite, the spin configuration $\bsi$ can be obtained by
setting minuses on one part of the graph and pluses on the other part.
Theorem \ref{theo1} below shows that, under hypothesis~\ref{adhy}, 
we cannot build a Markovian coupling which is reversible with respect to the $IP$ measure
and whose marginal on the spins 
is a one spin flip Markov jump process, i.e., 
with transition rates satisfying (\ref{tauxglauber}).
\begin{theo}\label{theo1}
Suppose that hypothesis~\ref{adhy} holds.
Let $(\eta_t,\sigma_t)_{t\geq 0}$ be a Markovian jump process with
rates $q((\eta,\sigma), (\eta',\sigma'))$  satisfying (\ref{qmatrix1}), (\ref{qmatrix})
and the detailed balance equation (\ref{eqreversible}). If the second marginal $(\sigma_t)_{t\geq 0}$
is a Markov jump process with transition rates satisfying (\ref{tauxglauber}),
then the first marginal $(\eta_t)_{t\geq 0}$ is a continuous-time non-Markovian jump process.
\end{theo}
\begin{proof}
Let $(\eta_t,\sigma_t)_{t\geq 0}$ be a Markovian jump process with transition
rates $q((\eta,\sigma), (\eta',\sigma'))$  satisfying (\ref{qmatrix1}), (\ref{qmatrix})
and the detailed
balance equation (\ref{eqreversible}). 
	Suppose also that
the first marginal $(\eta_t)_{t\geq 0}$
	is a Markov jump process with transition rates ${\tilde c}(\eta,\eta')$.
	Suppose also that
the second marginal $(\sigma_t)_{t\geq 0}$
is a Markov jump process with transition rates $c(\sigma,\sigma')$ satisfying (\ref{tauxglauber}).
From Ball and Yeo \cite{Ball} (see theorem \ref{theoball} of the appendix), we have 
	\begin{equation}
		\label{mar1}
	\forall \eta,\eta'
\in \{0,1\}^E\qquad
	\forall \sigma
\in \{-1,1\}^V\qquad
		\sum_{\sigma'\in \{-1,1\}^V} 
	q((\eta,\sigma),(\eta',\sigma'))= {\tilde c}(\eta,\eta')\,,
	\end{equation}
	\begin{equation}
		\label{mar2}
	\forall \sigma,\sigma'
\in \{-1,1\}^V\qquad
	\forall \eta
\in \{0,1\}^E\qquad
		\sum_{\eta'\in \{0,1\}^E} q((\eta,\sigma),(\eta',\sigma'))= {c}(\sigma,\sigma')\,.
\end{equation}
Let $\bsi$ be a configuration as in hypothesis~\ref{adhy} and let $\bx$ be a vertex of $E$ belonging
to at least four distinct edges.
Let $\bet$ be the configuration where all the edges are closed and let
	$\bet^{\bx}$ be the configuration where the edges exiting from $\bx$ are opened,
	while all the other edges are closed.
	From formula~\eqref{mar2} applied to $\bet,\bsi$ and $\bsi^{\bx}$, we have
	\begin{equation*}
		\sum_{\eta'\in \{0,1\}^E} q((\bet^\bx,\bsi^\bx),(\eta',\bsi))= 
		{c}(\bsi^\bx,\bsi)\,.
\end{equation*}
The only configuration $\eta'$ compatible with $\bsi$ is $\bet$, thus
the above identity reduces to
		$$q((\bet^\bx,\bsi^\bx),(\bet,\bsi))= 
		{c}(\bsi^\bx,\bsi)\,.$$
The conditions~\eqref{tauxglauber} implies that
		${c}(\bsi^\bx,\bsi)$ is positive, 
		therefore
		$q((\bet^\bx,\bsi^\bx),(\bet,\bsi))$ is positive. 
		Since we assumed that the dynamics is reversible, then the
		rate
		$q((\bet,\bsi),(\bet^\bx,\bsi^\bx))$ is also positive. 
	From formula~\eqref{mar1} applied to $\bet,\bet^\bx$ and $\bsi$, we have
	\begin{equation*}
		\sum_{\sigma'\in \{-1,1\}^V} 
	q((\bet,\bsi),(\bet^\bx,\sigma'))= {\tilde c}(\bet,\bet^\bx)\,,
\end{equation*}
therefore the rate
	${\tilde c}(\bet,\bet^\bx)$ is positive.
Let now $\hsi$
	be a spin configuration in which the spins of the
	vertices connected to $\bx$ contain at least two negative spins and at least two
	positive spins. This is possible because we assumed that $\bx$ is connected to 
	at least four distinct vertices.
	Applying again formula~\eqref{mar1}, this time to
$\bet,\bet^\bx$ and $\hsi$, we have
	\begin{equation*}
		\sum_{\sigma'\in \{-1,1\}^V} 
	q((\bet,\hsi),(\bet^\bx,\sigma'))= {\tilde c}(\bet,\bet^\bx)\,.
\end{equation*}
Since the rate
	${\tilde c}(\bet,\bet^\bx)$ is positive, then there exists 
	a spin configuration $\sigma'$ such that
	$$q((\bet,\hsi),(\bet^\bx,\sigma'))\,>\,0\,.$$
In the configuration $\bet^\bx$, all the edges emanating from $\bx$ are opened, 
hence the neighbours of $\bx$ are connected. Since
the configuration $\sigma'$ has to be compatible with $\bet^\bx$, then
all the 
neighbours of $\bx$ 
have the same spin in $\sigma'$, hence
the configuration $\sigma'$ has to differ from $\hsi$ in at least two vertices.
From 
the conditions~\eqref{tauxglauber}, we should therefore have that
$c(\hsi,\sigma')=0$.
Yet formula~\eqref{mar2} applied to $\bet,\bsi$ and $\bsi^{\bx}$ yields
	\begin{equation*}
		{c}(\hsi,\sigma')\,=\,
		\sum_{\eta'\in \{0,1\}^E} q((\bet,\hsi),(\eta',\sigma')) \,\geq\,
	q((\bet,\hsi),(\bet^\bx,\sigma'))\,>\,0\,,
\end{equation*}
which is contradictory.
\end{proof}
\noindent
Notice that hypothesis~\ref{adhy} holds for a cubic box on the $d$--dimensional 
lattice
when $d\geq 2$. Thus our dream process is not realizable in the lattice 
$\mathbb{Z}^d$ for $d\geq 2$.
\section{One edge and one incident cluster}
\label{three}
We drop here the requirement that the dynamics is local.
We build a dynamics which updates at each time
at most
one edge and the vertices belonging to the open cluster of one of the endpoints of the edge.
This dynamics will be
reversible with respect to the coupling measure $IP$
and its marginal on the edges is the FK dynamics described in 
section~\ref{fkdyna}.
For $x\in V$, recall that
$E_x$ is the set of the edges $e$ in $E$ having $x$ as endvertex, i.e.,
$$
E_x=\{e\in E,\, e=\langle x,y\rangle,\, y\in V\}\,.
$$
The following definition gives the transition rates of the dynamics.
\begin{defi}\label{def7}
Let $(\eta,\sigma)$, $(\eta',\sigma')$ be two elements of $\{0,1\}^E\times\{-1,1\}^V$.
We consider several cases:

\noindent
$\bullet$ If $\sigma'=\sigma$ and
	there exists $e\in E$ such that $\eta'=\eta^e$ and $\gamma_\eta(e)=1$, then we 
	set
$$
	c((\eta,\sigma), (\eta',\sigma'))= 
	\big(
	(1-p)\BBone_{\eta(e)=1}+
	p\BBone_{\eta(e)=0}\big)
\BBone_{(\eta,\sigma)\in \cC}
	\,.
$$

\noindent
$\bullet$ If $\sigma'=\sigma$ and
	there exists $e\in E$ such that $\eta'=\eta^e$, $\gamma_\eta(e)=0$ and
	$\delta_\sigma(e)=1$, then we set
$$
	c((\eta,\sigma), (\eta',\sigma'))\,= \,
	\frac{1}{2}
	\big(
	(1-p)\BBone_{\eta(e)=1}+
	p
	\BBone_{\eta(e)=0}\big)
\BBone_{(\eta,\sigma)\in \cC}
	\,.
$$

\noindent
$\bullet$ If there exists $x\in V$ and $e\in E_x$ such that
	$\eta'=\eta^e$, $\gamma_\eta(e)=0$, $\eta(e)=\delta_\sigma(e)$
	and 
	$\sigma'$ is given by
	$$\forall y\in V\qquad\sigma'(y)\,=\,
\begin{cases}
	\sigma(y)&\text{ if $y$ is not connected to $x$ in $\eta\setminus\{ e\}$}\,,\cr
	-\sigma(y)&\text{ if $y$ is connected to $x$ in $\eta\setminus\{ e\}$}\,,\cr
\end{cases}
$$
	then we set
$$
	c((\eta,\sigma), (\eta',\sigma'))\,= \,
	\frac{1}{4}
	\big(
	(1-p)\BBone_{\eta(e)=1}
\BBone_{(\eta,\sigma)\in \cC}
	+ p
	\BBone_{\eta(e)=0}
\BBone_{(\eta',\sigma')\in \cC}
	\big)
	\,.
$$

\noindent
$\bullet$ Otherwise, if
$(\eta,\sigma)\neq (\eta',\sigma')$, then we set
$c((\eta,\sigma), (\eta',\sigma'))= 0$.

\noindent
$\bullet$ Finally, if $(\eta,\sigma)= (\eta',\sigma')$, then we set
$$
c((\eta,\sigma), (\eta,\sigma))= -\sum_{(\eta',\sigma'),\, (\eta',\sigma')\neq (\eta,\sigma)}c((\eta,\sigma), (\eta',\sigma'))\,.
$$
\end{defi}
%
\noindent
We first check that the dynamics associated to these rates is reversible with respect to the measure
$IP$.
\begin{pro}\label{qroreversible}
The transition rates introduced in definition \ref{def7}  satisfy the   detailed balance equation  (\ref{eqreversible}).
\end{pro}
\begin{proof}
Let $(\eta,\sigma)$ and $(\eta',\sigma')$ be fixed in $\{0,1\}^E\times \{-1,1\}^V$.
	We have to prove (\ref{eqreversible}). We consider several cases, as in definition~\ref{def7}.
Suppose first that 
$\sigma'=\sigma$ and
	there exists $e\in E$ such that $\eta'=\eta^e$ and $\gamma_\eta(e)=1$.
	We have then
\begin{multline*}
IP(\eta,\sigma)c((\eta,\sigma),(\eta',\sigma))\,=\,\cr
 \kern-23pt
	\frac{1}{Z}
	\left(\prod_{f\in E}\left(p\BBone_{\eta(f)=1}\delta_{\sigma}(f)+ (1-p)\BBone_{\eta(f)=0} \right)\right)
	\big(
	(1-p)\BBone_{\eta(e)=1}+
	p\BBone_{\eta(e)=0}\big)
\BBone_{(\eta,\sigma)\in \cC}
	\cr
	\,=\,
	\frac{1}{Z}
	\left(\prod_{f\in E\setminus\{e\}}
	\left(p\BBone_{\eta(f)=1}+ (1-p)\BBone_{\eta(f)=0} \right)\right)
	(1-p)p\,
\BBone_{(\eta,\sigma)\in \cC}\,.
\end{multline*}
	Since $\eta'=\eta^e$ and $\gamma_\eta(e)=1$, then 
${(\eta,\sigma)\in \cC}$ if and only if
${(\eta',\sigma)\in \cC}$.
Also, the product in the last formula is the same for $\eta$ and $\eta'$, 
thus we obtain the detailed balance equation (\ref{eqreversible}).
Suppose next that 
$\sigma'=\sigma$ and
	there exists $e\in E$ such that $\eta'=\eta^e$, $\gamma_\eta(e)=0$ and
	$\delta_\sigma(e)=1$.
	We have then
\begin{multline*}
IP(\eta,\sigma)c((\eta,\sigma),(\eta',\sigma))\,=\,\cr
 \kern-23pt
	\frac{1}{Z}
	\left(\prod_{f\in E}\left(p\BBone_{\eta(f)=1}\delta_{\sigma}(f)+ (1-p)\BBone_{\eta(f)=0} \right)\right)
	\frac{1}{2}
	\big(
	(1-p)\BBone_{\eta(e)=1}+
	p\BBone_{\eta(e)=0}\big)
\BBone_{(\eta,\sigma)\in \cC}
	\cr
	\,=\,
	\frac{1}{Z}
	\left(\prod_{f\in E\setminus\{e\}}
	\left(p\BBone_{\eta(f)=1}+ (1-p)\BBone_{\eta(f)=0} \right)\right)
	\frac{1}{2}
	(1-p)p\,
\BBone_{(\eta,\sigma)\in \cC}\,.
\end{multline*}
	Since $\eta'=\eta^e$ and $\delta_\sigma(e)=1$, then 
${(\eta,\sigma)\in \cC}$ if and only if
${(\eta',\sigma)\in \cC}$.
Also, the product in the last formula is the same for $\eta$ and $\eta'$, 
thus we obtain the detailed balance equation (\ref{eqreversible}).
Suppose finally that there exists $x\in V$ and $e\in E_x$ such that
	$\eta'=\eta^e$, $\gamma_\eta(e)=0$, $\eta(e)=\delta_\sigma(e)$
	and 
	$\sigma'$ is obtained from $\sigma$ by reversing all the spins of the sites
	which are connected to $x$ by an open path in $\eta\setminus\{e\}$.
	We have then
$$\displaylines{
IP(\eta,\sigma)c((\eta,\sigma),(\eta',\sigma'))\,=\,\hfill\cr
	\frac{1}{Z}
	\left(\prod_{f\in E}\left(p\BBone_{\eta(f)=1}\delta_{\sigma}(f)+ (1-p)\BBone_{\eta(f)=0} \right)\right)
	\frac{1}{4}
	\big(
	(1-p)\BBone_{\eta(e)=1}
\BBone_{(\eta,\sigma)\in \cC}
	+ p
	\BBone_{\eta(e)=0}
\BBone_{(\eta',\sigma')\in \cC}
	\big)
	\cr
	\,=\,
	\frac{1}{Z}
	\left(\prod_{f\in E\setminus\{e\}}
	\left(p\BBone_{\eta(f)=1}+ (1-p)\BBone_{\eta(f)=0} \right)\right)
	\frac{1}{4}
	(1-p)p\,
	\big(
	\BBone_{\eta(e)=1}
\BBone_{(\eta,\sigma)\in \cC}
	+ 
	\BBone_{\eta(e)=0}
\BBone_{(\eta',\sigma')\in \cC}
	\big)
	\,.
	}$$
To remove the symbol $\delta_\sigma(f)$ in the last line,
	we have used the fact that, 
	if $\eta(e)=0$ and
	$(\eta',\sigma')\in \cC$, then
	$(\eta,\sigma)\in \cC$.
Since we have
$$
	\BBone_{\eta(e)=1}
\BBone_{(\eta,\sigma)\in \cC}
	+ 
	\BBone_{\eta(e)=0}
\BBone_{(\eta',\sigma')\in \cC}
\,=\,
	\BBone_{\eta'(e)=1}
\BBone_{(\eta',\sigma')\in \cC}
	+ 
	\BBone_{\eta'(e)=0}
\BBone_{(\eta,\sigma)\in \cC}\,,
$$
then we can conclude that
$$IP(\eta,\sigma)c((\eta,\sigma),(\eta',\sigma'))\,=\,
IP(\eta',\sigma')c((\eta',\sigma'),(\eta,\sigma))$$
and the detailed balance equation 
 holds also in this case.
We conclude that the detailed balance equation (\ref{eqreversible}) is always satisfied.
\end{proof}

\noindent We compute next the marginal dynamics on the edges.
\begin{pro}\label{qro2}
 Let $\eta$ be an element of $\{0,1\}^E$ and let $e\in E$.
 Let also $\sigma$ be an element of $\{-1,1\}^V$ such that 
$(\eta,\sigma)\in \cC$.

\noindent
$\bullet$ If $\gamma_\eta(e)=1$, then 
$$
\sum_{\sigma'\in \{-1,1\}^V}c((\eta,\sigma), (\eta^e,\sigma'))\,=\,
p\BBone_{\eta(e)=0}+ (1-p)\BBone_{\eta(e)=1}\,.
$$
\noindent
	$\bullet$ If $\gamma_\eta(e)=0$ and $\eta(e)=0$, then
$$
\sum_{\sigma'\in \{-1,1\}^V}c((\eta,\sigma), (\eta^e,\sigma'))\,=\,
	\frac{\displaystyle p}{\displaystyle 2}\,.
$$
\noindent
	$\bullet$ If $\gamma_\eta(e)=0$ and $\eta(e)=1$, then 
$$
\sum_{\sigma'\in \{-1,1\}^V}c((\eta,\sigma), (\eta^e,\sigma'))
	\,=\,
	1-p\,.
$$
\end{pro}
\begin{proof}
Let $\eta,e,\sigma$ be as in the statement of the proposition.
We use the expression for the rates given in
the definition~\ref{def7}.
If
$\gamma_\eta(e)=1$, then 
$c((\eta,\sigma), (\eta^e,\sigma'))$ is null unless $\sigma'=\sigma$. Thus the sum
reduces to 
	$c((\eta,\sigma), (\eta^e,\sigma))$, which is equal to $p$ if $\eta(e)=0$ and
	to $1-p$ if $\eta(e)=1$.
Suppose next that $\gamma_\eta(e)=0$ and $\eta(e)=0$. 
	We consider two further cases.
	If $\delta_\sigma(e)=1$, 
	then the sum
reduces again to 
	$c((\eta,\sigma), (\eta^e,\sigma))$, which in this case is equal to $p/2$.
	If $\delta_\sigma(e)=0$, 
	then the sum contains two terms, corresponding to the two spin configurations
	$\sigma'$
	obtained 
	from $\sigma$ by reversing all the spins of the sites
	which are connected to one extremity of $e$ by an open path in $\eta\setminus\{e\}$.
	These two rates
	$c((\eta,\sigma), (\eta^e,\sigma'))$ are both equal to $p/4$, and their sum is again
	equal to $p/2$.
	Suppose finally that
	$\gamma_\eta(e)=0$ and $\eta(e)=1$.
	The sum 
	contains three terms, one corresponding to $\sigma'=\sigma$
	and two
	corresponding to the two spin configurations
	$\sigma'$
	obtained 
	from $\sigma$ by reversing all the spins of the sites
	which are connected to one extremity of $e$ by an open path in $\eta\setminus\{e\}$.
	The rate
	$c((\eta,\sigma), (\eta^e,\sigma))$ is equal to $(1-p)/2$.
	The two other rates
	$c((\eta,\sigma), (\eta^e,\sigma'))$ are equal to $(1-p)/4$.
	In total, the sum is equal to $1-p$.
\end{proof}
\noindent
The marginal dynamics on the spins is quite complicated, but it is not a Markov process.
Let us consider an edge configuration $\eta$ and a spin configuration $\sigma$ such that
there exist $x\in V$ and $e\in E_x$ such that
	$\gamma_\eta(e)=\eta(e)=\delta_\sigma(e)=0$.
	Let 
	$\sigma'$ be the spin configuration defined by
	$$\forall y\in V\qquad\sigma'(y)\,=\,
\begin{cases}
	\sigma(y)&\text{ if $y$ is not connected to $x$ in $\eta\setminus\{ e\}$}\,,\cr
	-\sigma(y)&\text{ if $y$ is connected to $x$ in $\eta\setminus\{ e\}$}\,.\cr
\end{cases}
$$
We have then
$$
\sum_{\eta'\in \{0,1\}^E}c((\eta,\sigma), (\eta',\sigma'))\,=\,
	c((\eta,\sigma), (\eta^e,\sigma'))\,= \,
	\frac{p}{4}\,.
	$$
However the configuration $\sigma'$ depends on $\eta$. Indeed, if we change the status
of some edges 
$f$ in $E_x\setminus \{e\}$, we will change the set of the vertices which are connected 
to $x$ in $\eta\setminus\{e\}$, and we will obtain an edge configuration $\overline{\eta}$
for which
$$\sum_{\eta'\in \{0,1\}^E}c((\overline{\eta},\sigma), (\eta',\sigma'))\,=\,
c((\overline{\eta},\sigma), (\overline{\eta}^e,\sigma'))\,= \,0\,.$$
Thus the above sum depends on the configuration $\eta$ and the marginal process
on the spin configurations is not a Markov process.
\noindent 
This remark, together with propositions  \ref{qroreversible} and~\ref{qro2}, yield 
the following theorem.
\begin{theo}\label{theo3}
Let $(\eta_t,\sigma_t)_{t\geq 0}$ be a  Markov jump process on $\cC$ with the transition rates
introduced in definition \ref{def1}. Then $(\eta_t,\sigma_t)_{t\geq 0}$ is reversible with respect to the coupling measure IP.
	Its first marginal $(\eta_t)_{t\geq 0}$ is a 
	Markov process evolving according to the FK dynamics while 
its second marginal process
 $(\sigma_t)_{t\geq 0}$ is a 
	non-Markovian jump process.
\end{theo}
\begin{proof}
The reversibility property is deduced from Proposition  \ref{qroreversible}. Proposition \ref{qro2}
shows that the sum $\sum_{\sigma'\in \{-1,1\}^V}c((\eta,\sigma), (\eta',\sigma'))$ does not depend on $\sigma$ for any $\eta,\eta'\in \{0,1\}^E$,
moreover this sum is equal to the  transition rates given in (\ref{marginal}).
The proof of Theorem \ref{theo3} is complete thanks to the remark before the theorem and 
	the result of Ball and Yeo \cite{Ball}.
\end{proof}
So this third dynamics does not fulfill our dream, although it is not even local. 

\section{One edge and one spin}
\label{four}
A notable inconvenient of the previous dynamics is that it might reverse simultaneously
all the spins associated to an open cluster of the edge configuration.
We introduce here a slight modification of this dynamics to ensure that at most one spin
is changed at a time.
For $x\in V$, recall that
$E_x$ is the set of the edges $e$ in $E$ having $x$ as endvertex, i.e.,
$$
E_x=\{e\in E,\, e=\langle x,y\rangle,\, y\in V\}\,.
$$
The following definition gives the transition rates of the dynamics.
\begin{defi}\label{def8}
Let $(\eta,\sigma)$, $(\eta',\sigma')$ be two elements of $\{0,1\}^E\times\{-1,1\}^V$.
We consider several cases:

\noindent
$\bullet$ If $\sigma'=\sigma$ and
	there exists $e\in E$ such that $\eta'=\eta^e$ and $\gamma_\eta(e)=1$, then we 
	set
$$
	c((\eta,\sigma), (\eta',\sigma'))= 
	\big(
	(1-p)\BBone_{\eta(e)=1}+
	p\BBone_{\eta(e)=0}\big)
\BBone_{(\eta,\sigma)\in \cC}
	\,.
$$

\noindent
$\bullet$ If $\sigma'=\sigma$ and
	there exists $e\in E$ such that $\eta'=\eta^e$, $\gamma_\eta(e)=0$ and
	$\delta_\sigma(e)=1$, then we set
$$
	c((\eta,\sigma), (\eta',\sigma'))\,= \,
	\frac{1}{2}
	\big(
	(1-p)\BBone_{\eta(e)=1}+
	p
	\BBone_{\eta(e)=0}\big)
\BBone_{(\eta,\sigma)\in \cC}
	\,.
$$

\noindent
$\bullet$ If there exists $x\in V$ and $e\in E_x$ such that
	$\sigma'=\sigma^x$, $\eta'=\eta^e$, $\gamma_\eta(e)=0$, $\eta(e)=\delta_\sigma(e)$
	and $\eta(f)=0$ for $f\in E_x\setminus\{e\}$,
	then we set
$$
	c((\eta,\sigma), (\eta',\sigma'))\,= \,
	\frac{1}{4}
	\big(
	(1-p)\BBone_{\eta(e)=1}
\BBone_{(\eta,\sigma)\in \cC}
	+ p
	\BBone_{\eta(e)=0}
\BBone_{(\eta',\sigma')\in \cC}
	\big)
	\,.
$$

\noindent
$\bullet$ Otherwise, if
$(\eta,\sigma)\neq (\eta',\sigma')$, then we set
$c((\eta,\sigma), (\eta',\sigma'))= 0$.

\noindent
$\bullet$ Finally, if $(\eta,\sigma)= (\eta',\sigma')$, then we set
$$
c((\eta,\sigma), (\eta,\sigma))= -\sum_{(\eta',\sigma'),\, (\eta',\sigma')\neq (\eta,\sigma)}c((\eta,\sigma), (\eta',\sigma'))\,.
$$
\end{defi}
\noindent
The difference compared to
the rates of 
definition~\ref{def7} is that, in
definition~\ref{def8}, we allow to reverse the spins of a cluster only
when it is reduced
to a single vertex.
This dynamics possesses the same properties as the dynamics of the previous section, namely,
it is reversible with respect to the coupling measure IP,
	its first marginal process on the edges
is a 
	Markov process 
	while 
its second marginal process on the spins is a non-Markovian jump process.
The marginal dynamics on the edges differs from the FK dynamics in the following way.
An edge can be opened between two different clusters only if one the two clusters is reduced
to a single vertex.
More precisely, the corresponding rates are the following.
For $x\in V$ and $e\in E_x$, for any configuration $\eta$,
\begin{equation*}
	c(\eta,\eta^e)
\,=\,
\begin{cases}
	1-p&\text{ if }\eta(e)=1\,,\cr
	p&\text{ if }\eta(e)=0\text{ and }\gamma_\eta(e)=1\,,\cr
	p/2&\text{ if }\eta(f)=0\text{ for $f\in E_x$}\,.\cr
\end{cases}
\end{equation*}
Funnily enough, this fourth dynamics can be seen as a contraction of the first dynamics,
which could make only one change at a time. 
\noindent The following proposition shows that,
if the initial condition $(\eta_0,\sigma_0)$ is distributed as $IP$,
then
the infinitesimal behavior of the marginal on the spins of the process $(\eta_t,\sigma_t)_{t\geq 0}$,
whose rates are introduced in definition \ref{def8}, is
quite similar to the transition rates of the Glauber dynamics as given in (\ref{marginal}). The difference
lies in the factor $|E_x|$. On the cubic lattice $\mathbb{Z}^d$, the factor $|E_x|$ is equal to $2d$ and it
is independent of $x$.
\begin{pro}\label{umarg}
Let $(\eta_t,\sigma_t)_{t\geq 0}$ be a  Markov jump process with the transition rates
introduced in definition \ref{def8}. Suppose that  the pair $(\eta_0,\sigma_0)$ is distributed according to $IP$. Then, for any $x\in V$,
for any $\sigma\in \{-1,1\}^V$ and any $s\geq 0$,
\begin{equation}\label{ulimite}
\lim_{t\rightarrow 0}\frac{1}{t}\BBp(\sigma_{t+s}=\sigma^x|\sigma_s=\sigma)\,=\,
	\frac{1}{4}
	|E_x|
 p(1-p)^{|\{\,e\in E_x:\delta_\sigma(e)=1\,\}|}\,.
\end{equation}
\end{pro}
\begin{proof} We have, for any $\sigma\in \{-1,1\}^V$, $x\in V$ and $s\geq 0$,
\begin{eqnarray}\label{uee1}
&& \BBp(\sigma_{t+s}=\sigma^x|\sigma_s=\sigma) = \sum_{\eta\in \{0,1\}^E}\sum_{\eta'\in \{0,1\}^E}\BBp(\sigma_{t+s}=\sigma^x, \eta_{t+s}=\eta',\eta_s=\eta|\sigma_s=\sigma){\nonumber}\\
&&= \sum_{\eta\in \{0,1\}^E}\sum_{\eta'\in \{0,1\}^E}\BBp(\sigma_{t+s}=\sigma^x, \eta_{t+s}=\eta'|\sigma_s=\sigma, \eta_s=\eta) \BBp( \eta_s=\eta|\sigma_s=\sigma).
\end{eqnarray}
Recall that $(\eta_t,\sigma_t)_{t\geq 0}$ is a  Markov jump process with transition rates 
$c((\eta,\sigma),(\eta',\sigma'))$, thus
\begin{eqnarray}\label{uee2}
\lim_{t\rightarrow 0}\frac{1}{t}\BBp(\sigma_{t+s}=\sigma^x, \eta_{t+s}=\eta'|\sigma_s=\sigma, \eta_s=\eta)
\,=\, c((\eta,\sigma),(\eta',\sigma^x))\,.
\end{eqnarray}
Using the expression of the 
rates 
$c((\eta,\sigma),(\eta',\sigma'))$ given in
definition \ref{def8}, together with formulas~(\ref{uee1}) and (\ref{uee2}), we obtain that,
for any $s\geq 0$,
	\begin{multline}\label{try}
 \lim_{t\rightarrow 0}\frac{1}{t}\BBp(\sigma_{t+s}=\sigma^x|\sigma_s=\sigma)
	\,=\,
	\sum_{\eta\in \{0,1\}^E}
	\sum_{e\in E_x}\left(\prod_{f\in E_x\setminus\{e\}}\BBone_{\eta(f)=0}\right)
	\cr
	\times
	\frac{1}{4}
	\big(
	(1-p)\BBone_{\eta(e)=1}
\BBone_{(\eta,\sigma)\in \cC}
	+ p
	\BBone_{\eta(e)=0}
\BBone_{(\eta^e,\sigma^x)\in \cC}
	\big)
	\BBp( \eta_s=\eta|\sigma_s=\sigma)
	\,.
\end{multline}
The measure $IP$ is, by proposition \ref{prorev1}, a reversible  and a stationary measure for the process $(\eta_t,\sigma_t)_{t\geq 0}$.
 Hence, if $(\eta_0,\sigma_0)$ is distributed as $IP$, then, for any $s\geq 0$,  $(\eta_s,\sigma_s)$ is also distributed as $IP$. Consequently, under the hypothesis of proposition \ref{umarg}, we have
\begin{equation}\label{fracte}
\BBp( \eta_s=\eta|\sigma_s=\sigma)=\frac{\BBp( \eta_s=\eta,\sigma_s=\sigma)}{\BBp(\sigma_s=\sigma)}=\frac{IP(\eta,\sigma)}
{\sum_{\eta'\in \{0,1\}^E}IP(\eta',\sigma)}\,.
\end{equation}
Equation~\eqref{transf} yields that
\begin{equation}\label{fructe}
{\sum_{\eta'\in \{0,1\}^E}IP(\eta',\sigma)} \,=\,
\frac{1}{Z} (1-p)^{|\{\,e\in E:\delta_\sigma(e)=0\,\}|}\,.
\end{equation}
Reporting~\eqref{fracte} and~\eqref{fructe} in \eqref{try},
we conclude that
\begin{multline}
	\label{atx}
\lim_{t\rightarrow 0}\frac{1}{t}\BBp(\sigma_{t+s}=\sigma^x|\sigma_s=\sigma)
\, = \,
{Z} (1-p)^{-|\{\,e\in E:\delta_\sigma(e)=0\,\}|}\times
	\cr
	\sum_{\eta\in \{0,1\}^E}
	\sum_{e\in E_x}\left(\prod_{f\in E_x\setminus\{e\}}\BBone_{\eta(f)=0}\right)
	\frac{1}{4}
	\big(
	(1-p)\BBone_{\eta(e)=1}
\BBone_{(\eta,\sigma)\in \cC}
	+ p
	\BBone_{\eta(e)=0}
\BBone_{(\eta^e,\sigma^x)\in \cC}
	\big)
	IP(\eta,\sigma)
\,.
\end{multline}
We compute these sums as follows.
We have, for $e\in E_x$,
$$\displaylines{
	\sum_{\eta\in \{0,1\}^E}
	\left(\prod_{f\in E_x\setminus\{e\}}\BBone_{\eta(f)=0}\right)
	\BBone_{\eta(e)=1}
\BBone_{(\eta,\sigma)\in \cC}
	IP(\eta,\sigma)
\,=\, \hfill\cr
\hfill
\frac{1}{Z} 
(1-p)^{|E_x|-1
+|\{\,e\in E\setminus E_x:\delta_\sigma(e)=0\,\}|}
	p\,\delta_\sigma(e)\,,\cr
	\sum_{\eta\in \{0,1\}^E}
	\left(\prod_{f\in E_x\setminus\{e\}}\BBone_{\eta(f)=0}\right)
	\BBone_{\eta(e)=0}
\BBone_{(\eta^e,\sigma^x)\in \cC}
	IP(\eta,\sigma)
\,=\,\hfill\cr
\hfill\frac{1}{Z} 
(1-p)^{|E_x|
+|\{\,e\in E\setminus E_x:\delta_\sigma(e)=0\,\}|}
	\big(1-\delta_\sigma(e)\big)\,.
	}$$
Reporting the values of these sums in formula~\eqref{atx}, we obtain
\begin{multline}
	\label{ayx}
\lim_{t\rightarrow 0}\frac{1}{t}\BBp(\sigma_{t+s}=\sigma^x|\sigma_s=\sigma)
\, = \,
	\frac{1}{4}
 (1-p)^{-|\{\,e\in E:\delta_\sigma(e)=0\,\}|}\times
	\cr
	\sum_{e\in E_x}
	\Big(
(1-p)^{|E_x|
+|\{\,e\in E\setminus E_x:\delta_\sigma(e)=0\,\}|}
	p\,\delta_\sigma(e)
	+
p(1-p)^{|E_x|
+|\{\,e\in E\setminus E_x:\delta_\sigma(e)=0\,\}|}
	\big(1-\delta_\sigma(e)\big)\Big)\cr
	\,=\,
	\frac{1}{4}
 (1-p)^{-|\{\,e\in E:\delta_\sigma(e)=0\,\}|}
	\sum_{e\in E_x}
p(1-p)^{|E_x|
+|\{\,e\in E\setminus E_x:\delta_\sigma(e)=0\,\}|}
\cr
	\,=\,
	\frac{1}{4}
	\sum_{e\in E_x}
 p(1-p)^{|\{\,e\in E_x:\delta_\sigma(e)=1\,\}|}
	\,=\,
	\frac{1}{4}
	|E_x|
 p(1-p)^{|\{\,e\in E_x:\delta_\sigma(e)=1\,\}|}
	\,.
\end{multline}
%
This yields the desired 
result~\eqref{ulimite}.
\end{proof}

\section*{Appendix}

\noindent Let $(V,E)$ be a finite graph and let $(X_t)_{t\geq 0}=((\eta_t,\sigma_t))_{t\geq 0}$ be a
time--homogeneous continuous--time Markov process  with  state space $\Omega=\{0,1\}^E\times \{-1,1\}^V$ and with infinitesimal generator ${{Q}}$ defined as follows. For
any function $f$ defined on $\Omega$ with values in $\mathbb R$, we have
$$
{{Q}}f(\eta,\sigma)=\sum_{(\eta',\sigma')\in \Omega} {{q}}((\eta,\sigma),(\eta',\sigma')) \left(f(\eta',\sigma')-f(\eta,\sigma)\right),
$$
where  ${{q}}((\eta,\sigma),(\eta',\sigma'))$ are
the transition rates defined by, for $(\eta,\sigma)\neq (\eta',\sigma')$,
$$
{{q}}((\eta,\sigma),(\eta',\sigma'))= \lim_{t\rightarrow 0^+}\frac{1}{t} \BBp\left(X_t=(\eta',\sigma') | X_0=(\eta,\sigma) \right))\,,
$$
and  for $(\eta,\sigma)=(\eta',\sigma')$,
\begin{equation}\label{E}
{{q}}((\eta,\sigma),(\eta,\sigma))= - \sum_{(\eta',\sigma'):\, (\eta,\sigma)\neq (\eta',\sigma') }{{q}}((\eta,\sigma),(\eta',\sigma')).
\end{equation}
Our goal
is to present some conditions on the transition rates
${{q}}((\eta,\sigma),(\eta',\sigma'))$
under which $(\sigma_t)_{t\geq 0}$ or $(\eta_t)_{t\geq 0}$ are Markov processes. This is a lumpability problem discussed in Ball and Yao (1993). We first recall the following definition of lumpability.
\begin{defi}
Let $\Omega$ be a countable set, $\{\Omega_1,\cdots,\Omega_r\}$ be a partition of $\Omega$ and $h$ the
	function from $\Omega$ to $\{1,\cdots,r\}$
defined by $h(x) = j$ if $x \in \Omega_j$.
	An homogeneous Markov chain $(X_t)_{t\geq 0}$ with state space $\Omega$ is
lumpable with respect to the partition ${\Omega_1,\cdots,\Omega_r}$ if $(h(X_t))_{t\geq 0}$ is an
	homogeneous Markov
chain for every initial distribution $y_0$ on $\{1,\cdots,r\}$ and its transition rates do not depend on the
choice of the initial distribution $y_0$.
\end{defi}
\noindent Let us take the following partition of $\Omega$:
$$
\Omega\,= \bigcup_{\sigma\in \{-1,1\}^V}
\Omega_{\sigma}
\,,\qquad
\Omega_{\sigma}\,=\,\big\{\,(\eta,\sigma)\in \Omega: \eta\in \{0,1\}^E\,\big\}\,.
$$
In this case, the function $h$ is the projection from $\Omega$ to $\{-1,1\}^V$ defined by $h((\eta,\sigma))=\sigma$.
Our purpose is to obtain necessary and sufficient
conditions under which the Markov process $((\eta_t,\sigma_t))_{t\geq 0}$ is
lumpable with respect to the partition $(\Omega_{\sigma})_{\sigma \in \{-1,1\}^V}$. \\
\\
It is well known (see for instance Theorem 3.1 in Ball and Yao (1993)\footnote{Ball and Yao's result is stated under their condition 2.2 which is satisfied in the actual context.
}) that a Markov process is lumpable with respect to the partition $(\Omega_j)_{1\leq j\leq r}$ if and only if there exist
positive real numbers $\lambda_{i,j}$, $1\leq i,j\leq r$, such that
$$
\forall i\neq j\,,\quad
\forall \, x\in \Omega_i\,,\qquad
\sum_{y\in \Omega_j}q_{x,y}= \lambda_{i,j}\,.
$$
In order to simplify the reading of this paper, we adapt this  criterion of lumpability to the case of
the partition $(\Omega_{\sigma})_{\sigma \in \{-1,1\}^V}$ and the  Markov process $((\eta_t,\sigma_t))_{t\geq 0}$. We summarize it in the following theorem.
\begin{theo}\label{theoball}
 Let $((\eta_t,\sigma_t))_{t\geq 0}$ be a continuous-time Markov process  with finite state space $\{0,1\}^E\times \{-1,1\}^V$ and transition rates
 ${{q}}((\eta,\sigma),(\eta',\sigma'))$. Then the following statements are equivalent:
 \begin{enumerate}
 \item[(i)] For any $\sigma\neq \sigma'$, there exist positive numbers $c(\sigma,\sigma')$ such that
 $$
		  \forall\, \eta\in \{0,1\}^E\,,\qquad
 \sum_{\eta'\in \{0,1\}^E}q((\eta,\sigma), (\eta',\sigma'))= c(\sigma,\sigma')\,.
 $$
 \item[(ii)] The spin marginal $(\sigma_t)_{t\geq 0}$ of the Markov process $((\eta_t,\sigma_t))_{t\geq 0}$ is a Markov process for every initial distribution $\sigma_0$ on $\{-1,1\}^V$ and its transition rates $c(\sigma,\sigma')$ do not depend on the choice of the initial distribution of $\sigma_0$.
 \end{enumerate}
\end{theo}
\noindent The analog of Theorem  \ref{theoball} holds for the edge marginal process of $((\eta_t,\sigma_t))_{t\geq 0}$. For a proof, we refer to Theorem 3.1 in
Ball and Yeo \cite{Ball}. We also recall Theorem 3.2 of Ball and Yeo \cite{Ball}, that we adapt to our context.
\begin{theo}\label{theoball2}
Suppose that the transition rates $q((\eta,\sigma), (\eta',\sigma'))$ of the Markov process $(\eta_t,\sigma_t)_{t\geq 0}$ satisfy
$$
q((\eta,\sigma), (\eta',\sigma'))\neq 0
	\,\,\,\Longrightarrow\,\,\,
	(\eta,\sigma)=(\eta',\sigma')\,\,{\mbox{or}}\,\, (\eta=\eta'\,\,{\mbox{and}}\,\,\sigma\neq \sigma')
	\,\,{\mbox{or}}\,\, (\eta\neq \eta'\,{\mbox{and}}\,\,\sigma= \sigma').
$$
Then $(\eta_t)_{t\geq 0}$ and $(\sigma_t)_{t\geq 0}$ are both Markov processes if and only if they are mutually independent.
\end{theo}
\bibliographystyle{plain}
\bibliography{coupling}


\end{document}